%% file: Main_Westervelt.tex
\documentclass[11pt, twoside, leqno]{amsart}  
\usepackage{a4wide}
\input{preamble}

\title[Quasilinear wave  equations with fractional dissipation]{Limiting behavior of quasilinear wave equations \\[1mm] with fractional-type dissipation}


\author[B. Kaltenbacher, M. Meliani, and V. Nikoli\'{c}]{\small Barbara Kaltenbacher, Mostafa Meliani, and Vanja Nikoli\'{c}}
\address{  \small
	Department of Mathematics, 
	Alpen-Adria-Universit\"at Klagenfurt 
	\\ Universit\"atsstra\ss e 65--67, A-9020 Klagenfurt, Austria}
\email{barbara.kaltenbacher@aau.at}
\address{ 
	Department of Mathematics \\ 
	Radboud University   \\ 
	Heyendaalseweg 135,
	6525 AJ Nijmegen, The Netherlands}
\email{mostafa.meliani@ru.nl} 
\email{vanja.nikolic@ru.nl}

\nobibliography*
\begin{document}
\vspace*{4mm}
\begin{abstract}
\input{abstract}
\end{abstract}
\vspace*{-8mm}
\maketitle           
\input{Introduction}
\input{GeneralFluxLaws}
\input{uniformanalysis}
\input{Limit_delta_2ndorder}

\input{kernelverification}
\input{conclusion}
\input{appendix_semidiscrete}
\section*{Acknowledgments}
The work of the first author was supported by the Austrian Science Fund FWF under the grant DOC 78.

\bibliography{references}{}
\bibliographystyle{siam} 
\end{document}

%% file: preamble.tex
\usepackage{lipsum}
\usepackage{amsfonts}
\usepackage{graphicx}
\usepackage{epstopdf}
\usepackage[normalem]{ulem}
\usepackage{cancel}
\usepackage{calligra}
\usepackage{amsfonts,amsmath,amsthm,amssymb}
\usepackage{mathtools}
\usepackage{hyperref}
\usepackage{autonum}
\usepackage{hhline}
\usepackage{array}
\usepackage{bibentry}
\usepackage{pifont}
\newcommand{\cmark}{\ding{51}}%
\usepackage{diagbox}
\usepackage[table,x11names]{xcolor}
\usepackage{tcolorbox}
\usepackage{mdframed}
\usepackage{multicol}
\usepackage{graphicx}
\usepackage{subcaption}
\usepackage{moreverb}
\usepackage{bbm}
\usepackage[margin=1.29in]{geometry}
\usepackage{todonotes}
\usepackage{scalerel,amssymb}
\allowdisplaybreaks
\usepackage{mathrsfs}  
\usepackage{lineno}
\usepackage{todonotes}
\usepackage{tikz}
\usepackage{pgfplots}
\usetikzlibrary{arrows.meta}
\usepackage[numbers,sort&compress]{natbib}
\definecolor{mygreen}{HTML}{43a047}
\usepackage{doi}
\usepackage{adjustbox}
\usepackage{alphalph}
\usepackage{booktabs}
\usepackage{makecell}
\usepackage{appendix}
\def\Re{\textup{Re}}

\def\eps{\varepsilon}
\def\beps{\bar{\varepsilon}}
\def \CfrakKeps{C_{\frakKeps}}

\def\calM{\mathcal{M}}

\def\ulal{\underline{\mathfrak{m}}}
\def\olal{\overline{\mathfrak{m}}}

\newcommand{\Om}{\Omega}
\newcommand{\D}{\Delta}

\newcommand{\Dt}{\textup{D}_t}

\def\aaa{\mathfrak{m}}

\def\calU{\mathcal{U}}

\definecolor{darkgreen}{rgb}{0,0.5,0}

\newcommand{\ueps}{u^\varepsilon}
\newcommand{\uteps}{u_t^\varepsilon}
\newcommand{\utteps}{u_{tt}^\varepsilon}
\newcommand{\cAone}{C_{\textbf{A}_1}}
\newcommand{\cAtwo}{C_{\textbf{A}_2}}
\newcommand{\ut}{u_t}
\newcommand{\utt}{u_{tt}}

\newcommand{\ddt}{\frac{\textup{d}}{\textup{d}t}}

\newcommand{\Dal}{{\textup{D}}_t^\alpha}
\newcommand{\Doal}{{\textup{D}}_t^{1-\alpha}}

\newcommand{\dt}{\, \textup{d} t}
\newcommand{\ds}{\, \textup{d} s }
\newcommand{\dx}{\, \textup{d} x}

\newcommand{\domega}{\, \textup{d} \omega}
\newcommand{\dxs}{\, \textup{d}x\textup{d}s}

\newcommand{\inttO}{\int_0^t \int_{\Omega}}
\newcommand{\intt}{\int_0^t}
\newcommand{\intT}{\int_0^T}
\newcommand{\intO}{\int_{\Omega}}
\newcommand{\nLtwo}[1]{\|#1\|_{L^2(\Omega)}}

\newcommand{\R}{\mathbb{R}} 
\newcommand{\N}{\mathbb{N}} 
\newcommand{\Ltwo}{L^2(\Omega)}
\newcommand{\Lfour}{L^4(\Omega)}
\newcommand{\Linf}{L^\infty(\Omega)}
\newcommand{\Hone}{H^1(\Omega)}
\newcommand{\Htwo}{H^2(\Omega)}
\newcommand{\Hthree}{H^3(\Omega)}
\newcommand{\Honezero}{H_0^1(\Omega)}
\newcommand{\Honetwo}{{H_\diamondsuit^2(\Omega)}}
\newcommand{\Honethree}{{H_\diamondsuit^3(\Omega)}}

\newtheorem{theorem}{Theorem}
\newtheorem{lemma}{Lemma}
\newtheorem{proposition}{Proposition}
\newtheorem{corollary}[theorem]{Corollary}
\newtheorem*{assumption*}{Assumptions}

\newtheorem{remark}{Remark}
\numberwithin{lemma}{section}
\numberwithin{proposition}{section}
\numberwithin{theorem}{section}
\numberwithin{equation}{section}
\makeatletter
\newcommand{\leqnomode}{\tagsleft@true}
\newcommand{\reqnomode}{\tagsleft@false}
\makeatother
\newcommand{\bfq}{\boldsymbol{q}}
\newcommand{\bfmu}{\boldsymbol{\mu}}
\newcommand{\bfxi}{\boldsymbol{\xi}}
\newcommand{\calK}{\mathcal{K}}
\newcommand{\calB}{\mathcal{B}}
\newcommand{\frakKeps}{\mathfrak{K}_\varepsilon}
\newcommand{\frakK}{\mathfrak{K}}

\newcommand\Lconv{\ast}
\newcommand\Fconv{\star}

\newcommand{\TK}{\mathcal{T}}

\newcommand{\rt}{\tau_\theta}

\definecolor{grey}{rgb}{0.5,0.5,0.5}



%% file: abstract.tex
	In this work, we investigate a class of quasilinear wave equations of Westervelt type with, {in general, nonlocal-in-time} dissipation. They arise as models of nonlinear sound propagation through complex media with anomalous diffusion of Gurtin--Pipkin type. Aiming at minimal assumptions on the involved memory kernels -- which we allow to be weakly singular -- we prove the well-posedness of such wave equations in a general theoretical framework. In particular, the Abel fractional kernels, as well as Mittag-Leffler-type kernels, are covered by our results. The analysis is carried out uniformly with respect to the small involved parameter on which the kernels depend and which can be physically interpreted as the sound diffusivity or the thermal relaxation time. We then analyze the behavior of solutions as this parameter vanishes, and in this way relate the equations to their limiting counterparts. To establish the limiting problems, we distinguish among different classes of kernels and analyze and discuss all ensuing cases. 

%% file: Introduction.tex
\section{Introduction} \label{Sec:Introduction}
The classical theory of nonlinear sound propagation is based on employing the Fourier law of heat conduction
within the system of governing equations.  Using this law has drawbacks as it exhibits the following nonphysical feature: ``a disturbance at any point in the body is felt instantly at every other point"~\cite{gurtin1968general}. This property is often referred to as the paradox of infinite speed of propagation. \\
\indent  In this work, we investigate quasilinear acoustic equations that originate from the use of the general Gurtin--Pipkin flux law~\cite{gurtin1968general} in place of the Fourier one. These equations are given by
\begin{equation} \label{West_general}
	((1+2k\ueps)\uteps)_t-c^2 \Delta \ueps -  \frakKeps *\Delta \uteps = f
\end{equation}
and can be understood as  nonlocal generalizations of the classical Westervelt wave equation~\cite{westervelt1963parametric}. Here $\ueps=\ueps(x,t)$ represents the acoustic pressure, $c>0$ the speed of sound, $k \in \R$ the nonlinearity coefficient, and $f=f(x,t)$ is the sound source.  The function $\frakKeps=\frakKeps(t)$ is a memory kernel and $*$ denotes the Laplace convolution in time:
\[
(\frakKeps\Lconv \uteps)(t)=\int_0^t \frakKeps(t-s)\uteps(s)\ds.
\]
 The memory kernel $\frakKeps$ depends on the parameter $\varepsilon <<1$ which is in practice rather small, and  it is thus important to determine the limiting behavior of solutions to \eqref{West_general} as $\varepsilon \searrow 0$. Different forms of the kernel $\frakKeps$ arise in the literature and will be covered by our investigations. The choice 
 \begin{equation} \label{example_fractional_kernel}
 \frakKeps = \eps  \frac{1}{\Gamma(\alpha)} t^{\alpha-1}, \quad \alpha \in (0,1), 
 \end{equation}
where $\Gamma(\cdot)$ is the Gamma function, leads to the Westervelt equation with time-fractional damping of Djrbashian--Caputo-derivative type~\cite{Djrbashian:1966,Caputo:1967}, derived in~\cite{prieur2011nonlinear}. In this case, the parameter $\varepsilon$ has the physical meaning of the so-called sound diffusivity~\cite{lighthill1956viscosity}. \\
\indent If the underlying flux law involves time-fractional relaxation, the kernels in the resulting acoustic equations have the form
\begin{equation} \label{kernel_form_intro}
\frakKeps =  \,\left(\frac{\rt}{\eps}\right)^{a-b}\frac{1}{\eps^b}t^{b-1}E_{a,b}\left(-\left(\frac{t}{\eps}\right)^a\right), \quad a,b \in (0,1],
\end{equation}
where $E_{a, b}$ is the generalized Mittag-Leffler function; see~\eqref{def_MittagLeffler} below for its definition. Here $\rt$ is a scaling parameter, which we discuss in detail in Section~\ref{Sec:Gurtinpipkin_laws}. In this setting, $\varepsilon$ plays the physical role of the thermal relaxation time. In the physics literature, it is usually denoted by $\tau$.  \\
\indent The goal of the present work is to investigate wave equations \eqref{West_general} in terms of their unique solvability and then perform a singular limit analysis with respect to $\varepsilon$.  In particular, we wish to establish in which sense (and, possibly, at which rate) the solutions of \eqref{West_general}, supplemented by the initial conditions and homogeneous Dirichlet boundary data, converge to the limiting problem as $\eps \searrow 0$. We aim to perform the analysis  under minimally restrictive assumptions on the memory kernel $\frakKeps$ that will allow us to cover the kernels in \eqref{example_fractional_kernel} and \eqref{kernel_form_intro}. \\
\indent This analysis constitutes a challenging task because of the interplay of  quasilinear and nonlocal evolution, together with the non-restrictive assumptions imposed on the memory kernel $\frakKeps$. Unlike available results in the literature for the analysis of nonlinear wave equations with memory (see, e.g., \cite{conti2005singular,lasiecka2017global}), we do not require the kernel to be smooth on $[0,\infty)$ or non-negative. Additionally, the analysis of Westervelt equation must ensure that the leading coefficient does not degenerate; in other words, that
\[
 1+ 2k \ueps \geq \ulal >0
\]
holds uniformly in $\eps$. This task puts an additional strain on the analysis, as it requires obtaining $\eps$-uniform bounds on $\|\ueps\|_{L^\infty(\Linf)}$ and guaranteeing their smallness.
\subsection*{Related works and novelty} To the best of our knowledge, this is the first rigorous work dealing with the limiting behavior of quasilinear wave equations with dissipation of Djrbashian--Caputo-derivative-type.  \\ 
\indent The local-in-time nonlinear acoustic models, on the other hand, are by now well-studied. We refer to~\cite{kaltenbacher2009global, meyer2011optimal, kaltenbacher2011wellposedness} for the analysis of the Westervelt equation with strong damping: 
\begin{equation} \label{West_general_stronglydamped}
	((1+2k\ueps)\uteps)_t-c^2 \Delta \ueps - \eps \Delta  \uteps = 0
\end{equation}
on smooth bounded domains.  Note that \eqref{West_general_stronglydamped} can be cast within the family of equations in \eqref{West_general} by choosing $\frakKeps= \eps \delta_0$, where $\delta_0$ is the Dirac delta distribution. With strong damping in the equation,
 the corresponding Dirichlet boundary-value problem is known to be globally well-posed and the energy of the system decays exponentially with time; see~\cite[Theorems 1.2 and 1.3]{kaltenbacher2009global}.  In the inviscid case ($\eps=0$),  smooth solutions are expected to exist only until a certain propagation time, after which a blow-up occurs. The local well-posedness analysis of the inviscid Westervelt equation with Dirichlet boundary conditions follows by the results of \cite{dorfler2016local}. Numerical experiments indicating gradient blowup for inviscid nonlinear acoustic models can be found, for example, in~\cite[Ch.\ 5]{kaltenbacher2007numerical}.\\
 \indent The limiting behavior of the strongly damped Westervelt equation \eqref{West_general_stronglydamped} as $\eps \searrow 0$ has been studied in~\cite{kaltenbacher2022parabolic}, where sufficient conditions have been determined for solutions $\{\ueps\}_{\eps \in (0, \beps)}$ to strongly converge to the solution of the inviscid problem at a linear rate. \\
 \indent We mention that local acoustic models involving more general nonlinearities or higher-order terms have also received plenty of attention in the literature; see, for example,~\cite{mizohata1993global, tani2017mathematical, kaltenbacher2012well} and the review paper~\cite{kaltenbacher2015mathematics}. Third-order in time models with regular memory have also been extensively studied; see, e.g.,~\cite{lasiecka2017global} and the references given therein. \\
\indent The analysis of quasilinear models in nonlinear acoustics involving time-fractional evolution has been initiated only recently and many questions are open. The well-posedness analysis of the Westervelt equation with fractional dissipation involving the kernel given in \eqref{example_fractional_kernel} can be found in~\cite{kaltenbacher2022inverse}; the analysis in~\cite{baker2022numerical} allows for certain Mittag-Leffler kernels as well. We also point out the singular limit analysis for a class of nonlocal differential equations in~\cite{conti2006singular, conti2005singular}, from which we adopt certain ideas in the limit analysis with Mittag-Leffler kernels in Section~\ref{Sec:LimitingAnalysis}.
\subsection*{Main contributions} 
Our main results pertain to establishing the limiting behavior of solutions to the family of equations \eqref{West_general} on bounded domains, supplemented by initial conditions and Dirichlet boundary data, as $\eps \searrow 0$. As will be shown, the limiting behavior will depend on the form of the kernel $\frakKeps$ and its dependence on $\varepsilon$, and so we will distinguish different classes of kernels motivated also by the physical background of  sound propagation through media with nonlocal heat flux laws.\\
\indent In case of the family of kernels given by $\frakKeps= \eps \frakK$, where $\frakK$ satisfies suitable regularity and coercivity assumptions and does not depend on $\eps$ (this case covers \eqref{example_fractional_kernel}), in the limit one obtains the inviscid Westervelt equation:
\begin{equation} \label{West_inviscid}
	((1+2ku)\ut)_t-c^2 \Delta u = f.
\end{equation}
The solutions of \eqref{West_general} will then be shown to converge in the standard energy norm 
	\begin{equation}\label{energy_norm}
	\|u\|_{\textup{E}} := \left(\|\ut\|^2_{L^\infty(\Ltwo)}+ \|u\|^2_{L^\infty(\Hone)}\right)^{1/2}
\end{equation}
to the solution of the inviscid model at a linear rate; see Theorem~\ref{Thm:Limit} and Corollary~\ref{Corollary:Limit_epsK} for details.  For this result to hold, sufficient smoothness of initial data is needed, namely
\[
 (u_0, u_1) \in \{u_0 \in \Hthree \cap \Honezero: \Delta u_0 \vert_{\partial \Omega}=0 \} \times (\Htwo \cap \Honezero).
\]
Additionally smallness of data {and short final time are needed to ensure uniform well-posedness}; however, smallness of the initial conditions can be imposed in a lower-order norm than that of their regularity space; see Theorem~\ref{Thm:Wellp_2ndorder_nonlocal}. This limiting result significantly generalizes~\cite[Theorem 4.1]{kaltenbacher2022parabolic}, where the limiting behavior of the strongly damped Westervelt equation (obtained here by setting $\frakK=\eps \delta_0$ in \eqref{West_general}) has been studied.  
\\
\indent Remarkably, for Mittag-Leffler-type kernels, the limiting dissipation produces a richer family of limiting equations, as we rigorously discuss in Section~\ref{Sec:LimitingAnalysis}. For example, when $a \leq b$ in \eqref{kernel_form_intro}, the limiting equations are given by 
\begin{equation} 
\begin{aligned}
&	((1+2{k} u)\ut)_t-c^2 \Delta u -  \rt^{a-b}
{\textup{D}_t^{a-b+1}\D u} =f,	\end{aligned} 
\end{equation}
where $\textup{D}_t^{a-b+1}$ denotes the Djrbashian--Caputo fractional derivative of order $a-b+1$. The details of this result together with the convergence rate can be found in Theorem~\ref{Thm:Limit} and Proposition~\ref{Prop:Limit_a>b}.

\subsection*{Organization of the presentation} The rest of the presentation is organized as follows. In Section~\ref{Sec:AcousticModeling}, we motivate this study by discussing the physical background of these nonlocal models in the context of nonlinear acoustics and giving examples of relevant classes of memory kernels $\frakKeps$.  Section~\ref{Sec:UnifWellp_delta} tackles the question of the uniform well-posedness of the nonlinear model \eqref{West_general} under suitable regularity, uniform boundedness, and coercivity assumptions on the memory kernels. The results of this section form the basis for the rigorous study of the limiting behavior.  In Section~\ref{Sec:LimitingAnalysis}, we then establish the limiting behavior of \eqref{West_general} for different classes of $\frakKeps$. The main results of this part are contained in Theorem~\ref{Thm:Limit} and Propositions~\ref{Prop:Limit_a<=b} and~\ref{Prop:Limit_a>b}. Finally, we consider concrete dependencies of $\frakKeps$ on $\eps$ motivated by the physics of nonlinear acoustics and  prove in Section~\ref{sec:verification} that these classes of kernels indeed verify our theoretical assumptions. %
\subsection*{Notation} In the analysis, we use the notation $A\lesssim B$ for $A\leq C\, B$ when the constant $C>0$ does not depend on $\eps$.
 We denote by
\begin{equation} \label{sobolev_withtraces}
	\begin{aligned}
		\Honetwo=&\,H_0^1(\Omega)\cap H^2(\Omega), \\ 
		\Honethree=&\, \left\{u\in H^3(\Omega)\,:\, \mbox{tr}_{\partial\Omega} u = 0, \  \mbox{tr}_{\partial\Omega} \D u = 0\right\},
	\end{aligned}
\end{equation}
the	spaces of functions $H_0^2(\Omega)\subset \Honetwo\subset H^2(\Omega)$ and $H_0^3(\Omega)\subset \Honethree\subset H^3(\Omega)$ that satisfy boundary conditions given above. \\
\indent Given final time $T>0$ and $p, q \in [1, \infty]$, we use $\|\cdot\|_{L^p (L^q(\Omega))}$ to denote the norm on $L^p(0,T;L^q(\Omega))$ and $\|\cdot\|_{L^p_t (L^q(\Omega))}$ to denote the norm on $L^p(0,t;L^q(\Omega))$ for $t \in (0,T)$. We use $(\cdot, \cdot)_{L^2}$ for the scalar product on $\Ltwo$. \\
\indent We employ the following short-hand notation for the Abel kernel:
\begin{equation} \label{def_galpha}
	g_\alpha(t):= \frac{1}{\Gamma(\alpha)} t^{\alpha-1}, \quad \alpha \in (0,1) 
\end{equation}
and introduce the notational convention
\begin{equation} \label{def_g0}
	g_0:= \delta_0,
\end{equation}
where $\delta_0$ is the Dirac delta distribution. \\
\indent Throughout this work, $\Dt^{\eta}$  denotes the Djrbashian--Caputo fractional derivative, which is for $w \in W^{1,1}(0,t)$ defined by
\[
\Dt^{\eta}w(t)=g_{\lceil\eta\rceil - \eta} \Lconv \Dt^{\lceil\eta\rceil} w, \qquad -1<\eta <1;
\]
see, for example,~\cite[\S 1]{kubica2020time} and~\cite[\S 2.4.1]{podlubny1998fractional}.
Here $\lceil\eta\rceil \in\{0,1\}$ is the integer obtained by rounding up $\eta$ and $\Dt^{\lceil\eta\rceil}$ is the zeroth or first derivative operator.

%% file: GeneralFluxLaws.tex
\section{Physical motivation and examples of relevant classes of kernels} \label{Sec:AcousticModeling}
In this section, we motivate different classes of memory kernels $\frakKeps$ that will be covered by our analysis. These kernels arise in wave models of nonlinear sound propagation through media with nonlocal heat flux laws. The nonlocal flux laws are of the Gurtin--Pipkin type~\cite{gurtin1968general} given by
\begin{equation} \label{GP_flux_law}
	\bfq(t)= - \kappa  \, \frakK*\nabla \theta,
\end{equation} 
where $\boldsymbol{q}$ denotes the flux vector, $\theta$ the absolute temperature, and $\kappa$ is the thermal conductivity. As discussed in~\cite{povstenko2011fractional}, different choices of the kernel $\frakK$ lead to a rich family of flux laws that have appeared in the literature. 
\subsection*{Abel kernels} The choice
\begin{equation} \label{def_fractional_kernel}
\frakK = {\rt^{-\alpha}} g_{\alpha} 
\end{equation}
results in the flux law involving the fractional integral:
\begin{equation}\label{kernel_frac2}
	\begin{aligned}
		\bfq(t)= - \kappa {\rt^{-\alpha}} \frac{1}{\Gamma(\alpha)} \int_0^t (t-s)^{\alpha-1}\nabla \theta(s) \ds , \quad  0 < \alpha \leq 1;
	\end{aligned}
\end{equation}
see~\cite[Eq.\ (2.17)]{povstenko2015fractional} and \cite[Eq.\ (B.65)]{holm2019waves} for further modeling details.  The constant $\rt>0$ in \eqref{def_fractional_kernel} and \eqref{kernel_frac2} serves as a scaling factor to adjust for the dimensional inhomogeneity introduced by the fractional integral, in the way done in \cite[Appendix B.4.1.2]{holm2019waves}. We mention that some references account for this by changing the units of the thermal conductivity $\kappa$. 
\subsection*{Exponential kernels} Choosing an exponential kernel~\cite[Eq.\ (2.21)]{povstenko2015fractional}: 
\begin{equation} \label{exponential_kernel}
\frakK_\tau(t) = \frac{1}{\tau} \exp\left(-\frac{t}{\tau}\right) 
\end{equation}
leads to the heat flux law involving short-tail memory
\begin{equation}\label{kernel_exp}
	\begin{aligned}
		\bfq(t)=- \kappa\int_0^t \frac{1}{\tau} \exp\left(-\frac{t-s}{\tau}\right) \nabla \theta(s) \ds,
	\end{aligned}
\end{equation}
where $\tau <<1$ is the intrinsic relaxation time of the heat flux. Assuming $\bfq(0)=0$, \eqref{kernel_exp} can be seen as the solution to the well-known Maxwell--Cattaneo law~\cite{cattaneo1958forme}:
\begin{equation} \label{MC_law}
	\begin{aligned}
		\bfq+\tau \bfq_t= - \kappa \nabla \theta.
	\end{aligned}
\end{equation}
\subsection*{Mittag-Leffler-type kernels}\label{Sec:Gurtinpipkin_laws} Exponential kernels \eqref{exponential_kernel} are, in fact, just a particular case of a large family of Mittag-Leffler memory kernels. In nonlinear acoustics, these are motivated by the Compte--Metzler heat flux laws, put forward and investigated in~\cite{compte1997generalized}:
\begin{alignat}{3}
	\hspace*{-1.5cm}\text{\small(GFE I)}\hphantom{II}&& \qquad \qquad(1+\tau^\alpha \Dal)\boldsymbol{q}(t) =&&\, -\kappa {\rt^{1-\alpha}}\Doal \nabla \theta;\\[1mm]
	\hspace*{-1.5cm}\text{\small(GFE II)}\, \hphantom{I}&&\qquad \qquad (1+\tau^\alpha \Dal)\boldsymbol{q}(t) =&&\, -\kappa {\rt^{\alpha-1}}\Dt^{\alpha-1} \nabla \theta;\\[1mm]
	\hspace*{-1.5cm}\text{\small(GFE III)}\,\, && \qquad \qquad (1+\tau \partial_t)\boldsymbol{q}(t) =&&\, -\kappa {\rt^{1-\alpha}}\Doal \nabla \theta; \\[1mm]
	\hspace*{-1.5cm}\text{\small (GFE)}\hphantom{III}&&\qquad \qquad  (1+\tau^\alpha \Dal)\boldsymbol{q}(t) =&&-\kappa \nabla \theta.\hphantom{{\rt^{1-\alpha}}\Doal }
\end{alignat}
Here, as before, $\tau$ is the thermal relaxation time and we have introduced the scaling $\rt$ to ensure dimensional homogeneity so that $\kappa$ has the usual dimension of thermal conductivity.  These laws can be solved for $\bfq$ and rewritten in the Gurtin--Pipkin form, as noted in~\cite{povstenko2011fractional}:
\begin{equation}
	\begin{aligned}
		\bfq(t) = - \kappa \frakK_\tau * \nabla \theta
	\end{aligned}
\end{equation}
where the memory kernels are now given by
\begin{equation} \label{ML_kernels}
	\begin{aligned}
		\frakK_\tau = \left(\frac{\rt}{\tau}\right)^{a-b}\frac{1}{\tau^b}t^{b-1}E_{a,b}\left(-\left(\frac{t}{\tau}\right)^a\right).
	\end{aligned}
\end{equation}
We recall that the generalized Mittag-Leffler function is given by 
\begin{align} \label{def_MittagLeffler}
	E_{a, b}(t) = \sum_{k=0}^\infty \frac{t^k}{\Gamma(a k + b)}, \qquad a > 0,\ t,\, b \in \mathbb{R};
\end{align}
see, e.g.,~\cite[Ch.\ 2]{kubica2020time}. For the Compte--Metzler laws, the parameters $(a,b)$ should be chosen as in Table~\ref{tab:ker_par_mod},
 assuming additionally $\alpha>1/2$ for the first pair (law GFE I). \\
  \begin{center}
 	\begin{tabular}[h]{|c||c|c|c|c|}
 		\hline
 		~& GFE I & GFE II & GFE III & GFE \\
 		\hline \hline
 		$a$&$\alpha$ & $\alpha$ & $1$ & $\alpha$ \\
 		\hline
 		$b$	&$2\alpha-1$ & $1$ & $\alpha$ & $\alpha$ \\
 		\hline
 	\end{tabular}
 ~\\[2mm]
 	\captionof{table}{Parameters for the Mittag-Leffler kernels motivated by the Compte--Metzler laws}\label{tab:ker_par_mod}
 \end{center}
  Note that setting $a=b=1$ in \eqref{ML_kernels} allows us to cover the exponential kernel in \eqref{exponential_kernel}. 
\subsection*{Resulting acoustic models: Westervelt's equation and beyond} 
\indent The derivation of nonlinear acoustic models based on the governing system of equations of sound motion that involves the Gurtin--Pipkin flux law can be done following closely the steps taken in~\cite[\S 4]{jordan2014second}. The modeling at first leads to one of the two approximations of the original governing system:
the nonlocal wave equation of Blackstock type~\cite[p.\ 20]{blackstock1963approximate}
\begin{equation} \label{Blackstock_nonlocal}
	\psi_{tt}-c^2(1+2\tilde{k}\psi_t) \Delta \psi - \delta \frakK* \D\psi_{t}+ \ell\partial_t |\nabla \psi|^2=0
\end{equation}
or the nonlocal wave equation of Kuznetsov type~\cite{kuznetsov1971equations}
\begin{equation} \label{Kuznetsov_nonlocal}
	\begin{aligned}
		\begin{multlined}[t] (1+2\tilde{k}\psi_t)\psi_{tt}-c^2 \Delta \psi - \delta \frakK* \D\psi_{t}+ \ell\partial_t |\nabla \psi|^2=0,  \end{multlined}
	\end{aligned}
\end{equation}
where $\delta>0$ is the sound diffusivity and $\tilde{k}$, $\ell \in \R$. The equations are expressed in terms of the acoustic velocity potential $\psi=\psi(x,t)$, which is related to the pressure by
\[
u = \varrho \psi_t,
\] 
where $\varrho$ is the medium density.  If cumulative nonlinear effects in sound propagation are dominant so that approximation $|\nabla \psi|^2 \approx c^{-2}\psi_t^2$ holds,  a nonlocal version of the Westervelt equation~\cite{westervelt1963parametric} in potential form is obtained:
\begin{equation} \label{West_nonlocal}
	\begin{aligned}
		\begin{multlined}[t]
			(1+2\tilde{\tilde{k}}\psi_t)\psi_{tt}-c^2 \Delta \psi -  \delta \frakK* \D\psi_{t}=0,
		\end{multlined}
	\end{aligned}
\end{equation}
where $\tilde{\tilde{k}} = \tilde{k} +c^{-2} \ell$. 
Differentiating in time and using the relation $u = \varrho \psi_t$ leads to the equation in the scope of present work:
\begin{equation} \label{West_nonlocal_pressure}
	\begin{aligned}
		((1+2ku)u_{t})_t-c^2 \Delta u - \delta \frakK* \D u_{t}=  \varrho \delta \frakK(t) \Delta \psi_1,
	\end{aligned}
\end{equation}
with $k= \tilde{\tilde{k}}/\varrho$, where we have relied on the following differentiation formula:
\begin{equation} \label{West_first}
	(\frakK* \Delta \psi_t)_t =   \frakK* \Delta \psi_{tt} + \frakK(t) \Delta \psi_1.
\end{equation}
In case of the Mittag-Leffler kernel \eqref{ML_kernels}, we would simply have $\frakK_\tau$ above in place of $\frakK$. 

\subsection*{Unifying the physical parameters} Since we wish to investigate the limiting behavior in $\delta$ and in $\tau$ of the resulting Westervelt equation (where in case of $\tau \searrow 0$, we assume $\delta>0$ to be fixed) and the uniform well-posedness analysis in both cases is qualitatively the same, we unite $\delta$ and $\tau$ here into one parameter $\eps$ and consider the generalized equation:
\begin{equation}  
	\begin{aligned}
		((1+2k\ueps)\ueps_{t})_t-c^2 \Delta \ueps - \frakKeps* \D \uteps=f
	\end{aligned}
\end{equation}
with suitable assumptions to be made on $\frakKeps$ (see \eqref{assumption1} and \eqref{assumption2} below). Setting 
\begin{equation}
\frakKeps= \eps \frakK \quad \text{with } \ \frakK(t)= \rt^{-\alpha} g_{\alpha}(t)
\end{equation}
 will allow us to cover \eqref{West_nonlocal_pressure}, whereas choosing \begin{equation}
	\begin{aligned}
	\frakKeps
	=&\,  \delta \left(\frac{\rt}{\eps}\right)^{a-b} \frac{1}{\eps}\frakK\left(\frac{t}{\eps}\right) \quad \text{with } \ \frakK(t)= t^{b-1} E_{a,b}(-t^a)
	\end{aligned}
\end{equation}
 covers the setting of thermal relaxation. In the latter case, we set $\delta=1$ without loss of generality.  We analyze the equation with a general source term $f$ as the regularity assumptions needed for the analysis below are stronger than what we generally have with the right-hand side in \eqref{West_nonlocal_pressure} (unless one assumes that $\psi_1=0$). \\
\indent  We mention that using the Compte--Metzler flux laws stated as fractional ODEs in the derivation of acoustic equations is also possible but leads to qualitatively different, higher-order in time fractional acoustic models. We refer the interested reader to~\cite{kaltenbacher2022time} for their derivation and the well-posedness analysis.

%% file: uniformanalysis.tex
\section{ Uniform well-posedness analysis} \label{Sec:UnifWellp_delta}
The aim of this section is to establish the well-posedness of quasilinear Westervelt equation \eqref{ibvp_West_general}, uniformly in $\eps$. This result will be the basis for the later study of the limiting behavior. Throughout we assume $\Omega \subset \R^d$, where $d \in \{1, 2, 3\}$, to be a bounded and $C^3$ regular domain. We consider the following initial boundary-value problem:
\begin{equation}\label{ibvp_West_general}
	\left \{	\begin{aligned} 
		&((1+2k\ueps)\uteps)_t-c^2 \Delta \ueps -  \frakKeps *\Delta \uteps = f \quad  &&\text{in } \Omega \times (0,T), \\
		&\ueps =0 \quad  &&\text{on } \partial \Omega \times (0,T),\\
		&(\ueps, \uteps)=(u_0, u_1), \quad  &&\text{in }  \Omega \times \{0\}.
	\end{aligned} \right.
\end{equation}
Before proceeding to the question of solvability, we need to impose certain uniform regularity and coercivity assumptions on the memory kernel $\frakKeps$, which should be minimally restrictive. 
\subsection{Assumptions on the memory kernel in the analysis} \label{Sec:KernelAssumptions} Since we wish to study the limiting behavior of the nonlinear models as $\eps \searrow 0$, we may focus our attention in the  analysis  on an interval $(0, \beps)$ for some fixed $\beps>0$ without loss of generality. We make the following assumption on the kernel's regularity and boundedness.  \vspace*{2mm}

\begin{center}
\fbox{ 
	\begin{minipage}{0.9\textwidth}
		Given $\eps \in (0, \beps)$, the kernel satisfies
		\begin{equation}\label{assumption1} \tag{\ensuremath{\bf {{A}}_1}}
			\frakKeps \in L^1(0,T) \cup \{\varepsilon \delta_0\}  
		\end{equation}
		with the $\varepsilon$-uniform bound:
		\begin{equation}
			\|\frakKeps\|_{\calM(0,T)} \leq \cAone.
		\end{equation}
	\end{minipage}
}
\end{center}
~\\[1mm]

\indent Above we allow for $\frakKeps= \varepsilon\delta_0$, where $\delta_0$ is the Dirac delta distribution, to cover the case of having strong damping in the equation (i.e., $- \varepsilon \Delta \uteps$), although we primarily focus our attention in the analysis and presentation on $\frakKeps \in L^1(0,T)$. We use $\|\cdot\|_{\calM(0,T)}$ to denote the total variation norm which for the examples considered can be understood as:
\begin{equation} \label{def_calM}
\|\frakKeps\|_{\mathcal{M}(0,T)}=\begin{cases}
	\varepsilon &\text{ if }\frakKeps= \varepsilon \delta_0,\\
	\|\frakKeps\|_{L^1(0,T)}&\text{ if }\frakKeps \in L^1(0,T).
\end{cases}
\end{equation}
\indent  In the well-posedness analysis, we will also need a  coercivity assumption in order to achieve sufficient damping from the $\frakKeps$ term.  \vspace*{2mm}
\begin{center}
\fbox{ 
	\begin{minipage}{0.9\textwidth}
		For all $y\in L^2(0, T; \Ltwo)$, it holds that for all $t\in(0,T)$
		\begin{equation}\label{assumption2} \tag{\ensuremath{\bf {A}_{2}}}
			\begin{aligned}
				\int_0^{t} \intO \left(\frakKeps* y \right)(s) \,y(s)\dxs\geq 	C_{\frakKeps} 
				\int_0^{t} \|(\frakKeps* y)(s)\|^2_{\Ltwo} \ds, 
			\end{aligned}
		\end{equation}
		where the constant $C_{\frakKeps}>0$ is uniformly bounded from below:
		\begin{equation}
			C_{\frakKeps} \geq \cAtwo,
		\end{equation}
		where $\cAtwo>0$ does not depend on $\varepsilon \in(0,\beps)$.
	\end{minipage}
}
\end{center}
~\\[2mm]

\indent We mention that the two assumptions are verified for all classes of kernels discussed in the previous section, but postpone the proof of this claim to Section~\ref{sec:verification}.
\subsection*{Strategy in the well-posedness analysis}  To analyze \eqref{ibvp_West_general}, we introduce the fixed-point mapping
\begin{align}
\TK:\phi \mapsto \ueps,
\end{align}
where $\phi$ will belong to a ball in a suitable Bochner space and $\ueps$ will solve the linearized problem 
	\begin{equation} \label{ibvp_2ndorder_lin:Eq}
		\begin{aligned}
			(\aaa(\phi) \uteps)_t-c^2  \Delta \ueps -   \frakKeps*\D\uteps=f
		\end{aligned} 
	\end{equation}
	with the variable leading coefficient
	\begin{align} \label{coeffs_fp}
		\aaa(\phi)=1+2{k}\phi,
	\end{align}
	supplemented by the initial and boundary conditions given in \eqref{ibvp_West_general}.
Clearly, a fixed-point of this mapping $\phi=\ueps$ would solve the nonlinear problem.  This general strategy is common in the analysis of nonlinear acoustic models (see, e.g.,~\cite{kaltenbacher2009global}), however the crucial difference here is that it has to be conducted uniformly in $\eps$. 

\subsection{ Uniform well-posedness of a linear problem with variable principle coefficient} \label{Subsection:LinWellp}
\indent The well-definedness of the mapping as well as the proof of the existence of a unique fixed point rely on the uniform well-posedness of the linear problem, which we tackle first for general smooth and non-degenerate functions $\aaa=\aaa(\phi)$.
\begin{proposition} \label{Prop:WellP_Lin}
	Let $\eps \in (0, \beps)$. Given $T>0$, let 
	\begin{equation}
	\phi \in X_\phi:=L^\infty(0,T; \Honethree) \cap W^{1, \infty}(0,T; \Honetwo).
\end{equation}
 Assume that there exist $\overline{\aaa}$ and $\underline{\aaa}$, independent of $\eps$, such that 
\begin{equation} \label{non-deg_phi}
 0<\ulal \leq \aaa(\phi)=1+2k \phi(x,t)\leq \olal \quad \text{ in } \ \Omega \times (0,T). 
 \end{equation}
  Let  assumptions \eqref{assumption1} and \eqref{assumption2} on the kernel $\frakKeps$ hold.
	Furthermore, assume that the initial conditions satisfy
	\begin{equation} \label{IC_regularity}
		(u_0, u_1) \in \Honethree \times \Honetwo
	\end{equation}
	and let $f \in H^1(0,T; \Honezero)$.
	Then there exists a unique  solution of the linear problem
	\begin{equation}\label{ibvp_West_general_lin}
		\left \{	\begin{aligned} 
			&(\aaa(\phi)\uteps)_t-c^2 \Delta \ueps - \Delta \frakKeps * \uteps = f \quad  &&\text{in } \Omega \times (0,T), \\
			&\ueps =0 \quad  &&\text{on } \partial \Omega \times (0,T),\\
			&(\ueps, \uteps)=(u_0, u_1), \quad  &&\text{in }  \Omega \times \{0\},
		\end{aligned} \right.
	\end{equation}
	 such that
	\begin{equation}\label{regularity}
		\begin{aligned}
			\ueps \in \, \calU =& \,\begin{multlined}[t] L^\infty(0,T; \Honethree) \cap W^{1, \infty}(0,T; \Honetwo) 
				\cap H^2(0,T; \Honezero).
			\end{multlined}
		\end{aligned}
	\end{equation}
The solution fulfills the following estimate:
	\begin{equation} \label{final_est_lin}
		\begin{aligned}
			&\|\Delta \uteps\|_{L^\infty(\Ltwo)}^2+ \|\nabla\D \ueps\|_{L^\infty(\Ltwo)}^2   \\
			\leq&\, C(\phi, T)\,
			(	\nLtwo{\Delta u_1}^2+ \nLtwo{\nabla\D u_0}^2+\|f\|^2_{H^1(\Hone)}),
		\end{aligned}
	\end{equation}
	where the constant has the form
	\begin{equation}
		C(\phi, T) = C_1 \exp( C_2(1+\|\phi\|_{X_\phi}+\ldots +\|\phi\|^6_{X_\phi}) T)
	\end{equation}
	for some $C_{1}$, $C_2>0$ that do not depend on $\eps$.
\end{proposition}
\begin{proof} 
	We perform the analysis using a smooth Galerkin discretization in space; see, e.g.,~\cite[\S 7]{evans2010partial} and~\cite[Sec.\  3]{kaltenbacher2022parabolic} for details on this method. The existence of a unique approximate solution follows by reducing the semi-discrete problem to a system of Volterra integral equations of the second kind and employing relatively standard existence arguments, which we therefore provide in Appendix~\ref{Appendix:Semi-discrete}.\\
\indent	The key to the proof is a uniform energy bound with respect to the discretization that must, in this context, also be uniform in $\eps$. We follow the general testing strategy in \cite{kaltenbacher2022parabolic, kaltenbacher2022inverse} with special attention paid to the uniformity in $\eps$.  For notational simplicity, we drop the discretization parameter when denoting the approximate solution below.\\
\indent As we cannot rely much on the dissipation properties of the nonlocal term in the equation, we need sufficient smoothness of the solution to overpower the right-hand side terms arising after testing. For this reason, we test the (semi-discrete) problem with $\Delta^2 \uteps$. After integrating over space, we arrive at
	\begin{equation} \label{identity_1}
		\begin{aligned} 
			(\aaa \utteps -c^2\Delta \ueps-  \frakKeps*\Delta \uteps  + \aaa_t \uteps,\,  \Delta^2 \uteps)_{L^2} = (f, \Delta^2 \uteps)_{L^2}.
		\end{aligned}
	\end{equation}
	We note that the following identity holds:
	\begin{equation}
		\begin{aligned}
			(\aaa \utteps, \Delta^2 \uteps)_{L^2} 
				=&\,  (\aaa\D \utteps, \Delta \uteps)_{L^2}+(\utteps\, \D\aaa+2\nabla \utteps \cdot\nabla\aaa, \Delta \uteps)_{L^2} \\
			=&\,  \frac12\ddt(	\aaa\D \uteps, \Delta \uteps)_{L^2}-\frac12(\aaa_t\D \uteps, \Delta \uteps)_{L^2}+(\utteps \, \D\aaa+2\nabla \utteps \cdot\nabla\aaa, \Delta \uteps)_{L^2} 
		\end{aligned}
	\end{equation}
	because $\utteps \vert_{\partial \Omega} =\Delta \uteps \vert_{\partial \Omega}=0$.  Similarly, we have
		\begin{equation}
		\begin{aligned}
			(  \aaa_t \uteps, \Delta^2 \uteps )_{L^2}=&\,  (\D ( \aaa_t \uteps), \Delta \uteps )_{L^2}\\
			=&\, ( \aaa_t \D \uteps +2 \nabla \aaa_t\cdot \nabla \uteps +  \uteps \D \aaa_t, \Delta \uteps )_{L^2}.
		\end{aligned}
	\end{equation}
	Thus, by integrating by parts in \eqref{identity_1} in space, integrating over $(0,t)$ for $t \in (0,T)$, and employing coercivity assumption \eqref{assumption2}, we arrive at the following energy inequality:
	\begin{equation} \label{energy_ineq_1}
		\begin{aligned}
			&\frac12\nLtwo{\sqrt{\aaa} \Delta \uteps}^2 \Big \vert_0^t + \frac{c^2}{2} \nLtwo{\nabla\D \ueps}^2 \Big \vert_0^t +
			\cAtwo \int_0^{t} \|(\frakKeps*\nabla \D \uteps )(s)\|^2_{\Ltwo} \ds \\
			\leq&\,\begin{multlined}[t] 
				\frac12 \int_0^t (\aaa_t\D \uteps, \Delta \uteps)_{L^2}\ds+ \intt(\utteps \, \D\aaa+2\nabla \utteps \cdot\nabla\aaa, \Delta \uteps)_{L^2} \ds \\
				+ \intO \nabla f \cdot \nabla \D \ueps \dx \Big \vert_0^t - \inttO \nabla f_t \cdot \nabla \Delta \ueps \dxs,
 			\end{multlined}
		\end{aligned}
	\end{equation}
	where we have also used that $f|_{\partial \Om} = 0.$
	We can estimate the first integral on the right-hand side by relying on the embedding $\Htwo \hookrightarrow \Linf$ as follows:
	\begin{equation} \label{est_aaa_t}
		\begin{aligned}
			\frac12 \int_0^t (\aaa_t\D \uteps, \Delta \uteps)_{L^2}\ds
		\leq&\, |k| \|\phi_t\|_{L^\infty(\Htwo)}  \|\Delta \uteps\|^2_{L^2_t(\Ltwo)} \\
		\lesssim&\, \|\phi\|_{X_\phi}  \|\Delta \uteps\|^2_{L^2_t(\Ltwo)} ,
		\end{aligned}
	\end{equation}
	where the resulting term will be handled by Gr\"onwall's inequality later on. 
	We can use H\"older's and Young's inequalities to estimate the $f$ terms in \eqref{energy_ineq_1}:
	\begin{equation} \label{estimates_f}
		\begin{aligned}
		&
		\intO \nabla f \cdot \nabla \D \ueps \dx \Big \vert_0^t - \inttO \nabla f_t \cdot \nabla \Delta \ueps \dxs\\
		\leq&\, \begin{multlined}[t] \frac{1}{c^2}\|\nabla f(t)\|^2_{\Ltwo}+\frac{c^2}{4} \|\nabla \Delta \ueps(t)\|_{\Ltwo}^2+\frac12 \|\nabla f(0)\|^2_{\Ltwo}+ \frac12\|\nabla \Delta u_0\|_{\Ltwo}^2 \\+ 	\|\nabla f_t\|^2_{L^2(\Ltwo)}+  \|\nabla \Delta \ueps\|_{L^2_t(\Ltwo)}^2, \end{multlined}
		\end{aligned}
	\end{equation}
	where the two resulting $\nabla \D \ueps$ terms above can be either absorbed by the left-hand side of \eqref{energy_ineq_1} or handled by Gr\"onwall's inequality. \\
\indent	It remains to estimate the $\utteps$ terms on the right-hand side of \eqref{energy_ineq_1}. To this end, we first use H\"older's and Young's inequalities:
	\begin{equation} \label{est_utteps}
		\begin{aligned}
			&\intt(\utteps \, \D\aaa+2\nabla \utteps \cdot\nabla\aaa, \Delta \uteps)_{L^2} \ds \\
			\leq&\,  \left(\|\utteps\|_{L^2(\Lfour)}\|\Delta \aaa\|_{L^\infty(\Lfour)} + \|\nabla \utteps\|_{L^2(\Ltwo)}\|\nabla \aaa\|_{L^\infty(\Linf)} \right)\|\Delta \uteps\|_{L^2(\Ltwo)}.
		\end{aligned}
	\end{equation}
	By the embeddings $\Hone \hookrightarrow \Lfour$ and $\Htwo \hookrightarrow \Linf$, we further have
	\begin{equation}
		\begin{aligned}
		\|\Delta \aaa\|_{L^\infty(\Lfour)} +	\|\nabla \aaa\|_{L^\infty(\Linf)} \lesssim 	\|\Delta \phi\|_{L^\infty(\Hone)} +	\|\nabla \phi\|_{L^\infty(\Htwo)} \lesssim \|\phi\|_{X_\phi}
		\end{aligned}
	\end{equation}
	and thus
		\begin{equation} \label{est_utt_interim}
		\begin{aligned}
			\intt(\utteps \, \D\aaa+2\nabla \utteps \cdot\nabla\aaa, \Delta \uteps)_{L^2} \ds 
			\leq\,   \|\phi\|_{X_\phi} \|\utteps\|_{L^2(\Hone)} \|\Delta \uteps\|_{L^2(\Ltwo)}.
		\end{aligned}
	\end{equation}
	From here, we can use the (semi-discrete) PDE to further bound the $\utteps$ term. We first have by Poincar\'e's inequality
	\begin{equation}
	\begin{aligned}
		\|\utteps\|_{L^2(\Hone)}  \lesssim \| \nabla \utteps\|_{L^2(\Ltwo)}.
	\end{aligned}
\end{equation}
Then to estimate the right-hand side term, we use the identity
	\begin{equation} 
		\begin{aligned}
			\nabla \utteps
			=&\, \begin{multlined}[t]\nabla [\aaa^{-1}]\Big(  \frakKeps*\Delta \uteps +c^2\Delta \ueps - \aaa_t \uteps+ f\Big)+ \aaa^{-1}\Big(  \frakKeps*\nabla \Delta \uteps +c^2\nabla \Delta \ueps\\-  \nabla [\aaa_t \uteps]+\nabla f\Big).
				\end{multlined}
		\end{aligned}
	\end{equation}
The uniform boundedness of $\aaa$ in \eqref{non-deg_phi} allows us to conclude that
	\begin{equation}
	\begin{aligned}
		\|\nabla  \utteps\|_{L^2(\Ltwo)} 
		\lesssim&\,  \begin{multlined}[t] 
			\|\nabla \aaa\|_{L^\infty(\Linf)}\| \frakKeps*\Delta \uteps +c^2\Delta \ueps - \aaa_t \uteps+ f\|_{L^2(\Ltwo)}\\
		 + \| \frakKeps*\nabla \Delta \uteps +c^2\nabla \Delta \ueps-  \nabla [\aaa_t \uteps]+\nabla f\|_{L^2(\Ltwo)}.
		\end{multlined}
	\end{aligned}
\end{equation}
Further estimating the right-hand side terms leads to
	\begin{equation}
		\begin{aligned}
			&\|\nabla  \utteps\|_{L^2(\Ltwo)} \\
			\lesssim&\,  \begin{multlined}[t] 
			\| \phi\|_{X_\phi} \Big ( \| \Delta \ueps\|_{L^2(\Ltwo)}+  \|\frakKeps\|_{\calM(0,t)} \|\D \uteps \|_{L^2(\Ltwo)}+\| \phi\|_{X_\phi} \|\uteps\|_{L^2(\Ltwo)}+\|f \|_{L^2(\Ltwo)}\Big) \\
				+ \| \frakKeps*\nabla \Delta \uteps\|_{L^2(\Ltwo)} + \| \nabla \Delta \ueps\|_{L^2(\Ltwo)}+ 
				\| \phi\|_{X_\phi} \|\uteps\|_{L^2(\Lfour)}\\+ 
				\| \phi\|_{X_\phi} 
				\|\nabla \uteps\|_{L^2(\Ltwo)}+\|\nabla f\|_{L^2(\Ltwo)}.
			\end{multlined}
		\end{aligned}
	\end{equation}
	Thanks to assumption \eqref{assumption1}, we have the uniform bound $\|\frakKeps\|_{\calM(0,t)} \leq \cAone$.
Going back to \eqref{est_utt_interim} and using the estimate on $\nabla \utteps$ thus yields
		\begin{equation} \label{est_utt_final}
		\begin{aligned}
			&\intt(\utteps \, \D\aaa+2\nabla \utteps \cdot\nabla\aaa, \Delta \uteps)_{L^2} \ds \\
		\lesssim&\,  \begin{multlined}[t]  ( \|\phi\|_{X_\phi}+\|\phi\|_{X_\phi}^2+\|\phi\|_{X_\phi}^3) \Big\{ \| \Delta \ueps\|_{L^2(\Ltwo)}+   \|\D \uteps \|_{L^2(\Ltwo)}\\
			+ \| \frakKeps*\nabla \Delta \uteps\|_{L^2(\Ltwo)} + \|\uteps\|_{L^2(\Hone)}+\| f\|_{L^2(\Hone)}\Big\} \|\Delta \uteps\|_{L^2(\Ltwo)}.
				\end{multlined}
		\end{aligned}
	\end{equation}
By employing estimates \eqref{est_aaa_t}, \eqref{estimates_f}, and \eqref{est_utt_final}  in \eqref{energy_ineq_1}, we arrive at
		\begin{equation} \label{energy_ineq_2}
		\begin{aligned}
			&\frac12\nLtwo{\sqrt{\aaa} \Delta \uteps}^2 \Big \vert_0^t + \frac{c^2}{2} \nLtwo{\nabla\D \ueps}^2 \Big \vert_0^t +\cAtwo \int_0^{t} \|(\frakKeps*\nabla \D \uteps )(s)\|^2_{\Ltwo} \ds \\
			\lesssim &\,\begin{multlined}[t] (1+\|\phi_t\|_{X_\phi} ) \|\Delta \uteps\|^2_{L^2_t(\Ltwo)} +\|\nabla \Delta u_0\|^2_{\Ltwo}+ \|f\|^2_{{H^1(\Hone)}}\\
		+	( \|\phi\|_{X_\phi}+\|\phi\|_{X_\phi}^2+\|\phi\|_{X_\phi}^3) ^2\|\Delta \uteps\|^2_{L^2_t(\Ltwo)} \\+  \| \Delta \ueps\|^2_{L^2_t(\Ltwo)}+   \|\D \uteps \|^2_{L^2_t(\Ltwo)} + \|\uteps\|^2_{L^2_t(\Hone)} \\
		+	 ( \|\phi\|_{X_\phi}+\|\phi\|_{X_\phi}^2+\|\phi\|_{X_\phi}^3)\| \frakKeps*\nabla \Delta \uteps\|_{L^2_t(\Ltwo)}  \|\Delta \uteps\|_{L^2_t(\Ltwo)}.
						\end{multlined}
		\end{aligned}
	\end{equation}
	Other than the last term in \eqref{energy_ineq_2}, all other terms on the right-hand side can be tackled using Gr\"onwall's inequality. To treat the last term, we employ Young's inequality:
	\begin{equation} \label{est_Keps_gamma}
		\begin{aligned}
			&( \|\phi\|_{X_\phi}+\|\phi\|_{X_\phi}^2+\|\phi\|_{X_\phi}^3) \|  \frakKeps*\nabla \Delta \uteps\|_{L^2(\Ltwo)}  \|\Delta \uteps\|_{L^2(\Ltwo)} \\
			\leq&\,  \gamma	\| \nabla \Delta \frakKeps* \uteps\|^2_{L^2_t(\Ltwo)}+\frac{1}{4 \gamma } ( \|\phi\|_{X_\phi}+\|\phi\|_{X_\phi}^2+\|\phi\|_{X_\phi}^3)^2\|\Delta \uteps\|^2_{L^2_t(\Ltwo)} .
		\end{aligned}
	\end{equation}
If $\tilde{C}>0$ is the hidden constant within $\lesssim$ in \eqref{energy_ineq_2}, we can choose $\gamma$ as
	\[
	\gamma = \frac{1}{\tilde{C}}\cAtwo/2.
	\]
The term $\gamma	\| \nabla \Delta \frakKeps* \uteps\|^2_{L^2_t(\Ltwo)}$ can be absorbed by the left-hand side of \eqref{energy_ineq_2}. 
We then use Gr\"onwall's inequality for the second term on the right-hand side of \eqref{est_Keps_gamma} to arrive at an $\eps$-uniform estimate. \\
	\indent Altogether, starting from estimates \eqref{energy_ineq_2} and \eqref{est_Keps_gamma}, an application of Gr\"onwall's inequality leads to the final estimate \eqref{final_est_lin}, at first in a semi-discrete setting.  The bound transfers to the continuous setting through standard weak compactness arguments, analogously to, for example, \cite[Proposition 3.1]{kaltenbacher2022parabolic}. Uniqueness of the constructed solution can be established by testing the homogeneous problem (where $f=u_0=u_1=0$) with, for example, $\uteps$. We omit these details here. 
\end{proof}
\subsection{Uniform lower-order estimate}\label{Sec:Lower_order_est}
\indent In the fixed-point analysis, we will rely on Agmon's interpolation inequality to prove that the leading factor $\aaa(\ueps)=1+2k\ueps$ of the nonlinear equation does not degenerate. For $v \in \Htwo$, the inequality is given by
\begin{equation}\label{unif_bound_energynorm}
	\begin{aligned}
		\|v\|_{\Linf} \leq C_{\textup{A}}\|v\|_{\Ltwo}^{1-d/4}\|v\|^{d/4}_{\Htwo} \quad  (d \leq 3),
	\end{aligned}
\end{equation}
see~\cite[Lemma 13.2, Ch.\ 13]{agmon2010lectures}.  Using this estimate on $v=\ueps(t)$ together with a bound on $\|\ueps(t)\|_{\Ltwo}$ in terms of data in lower-order topology, will allow us to impose data smallness in that topology instead $\Honethree \times \Honetwo$.  We derive this bound next. \\
\indent Under the assumptions of Proposition~\ref{Prop:WellP_Lin}, testing \eqref{ibvp_West_general_lin} with $\uteps$ and using coercivity assumption \eqref{assumption2}  leads to the inequality                                                                                                                                                                                                                                                                                                                                                                                                                                                                                                                                                                                                     
\begin{equation} \label{lin_lower_bound}
	\begin{aligned}
		&\begin{multlined}[t]\frac12 \|\sqrt{\aaa}\uteps(t)\|^2_{\Ltwo} \Big \vert_0^t +\frac{c^2}{2}\|\nabla \ueps(t)\|^2_{\Ltwo}  \Big \vert_0^t+\cAtwo \int_0^{t} \|(\frakKeps*\nabla  \uteps )(s)\|^2_{\Ltwo} \ds 
			\end{multlined}
			\\
		\leq &              \frac12 \int_0^t (\aaa_t \uteps, \uteps)_{L^2}\ds+\int_0^t(f, \uteps)_{L^2}\ds\\
		\lesssim &\, \|\phi_t\|_{L^\infty(\Linf)}\|\uteps\|^2_{L^2_t(\Ltwo)}+\|f\|_{L^1(\Ltwo)}\|\uteps\|_{L^\infty_t(\Ltwo)},
	\end{aligned}
\end{equation}
where the hidden constant does not depend on $\eps$. From here by using Young's and Gr\"onwall's inequalities, we obtain the following lower-order counterpart of \eqref{final_est_lin}:
\begin{equation} \label{unif_est_lowerorder}
	\begin{aligned}
	&\|\uteps\|^2_{L^\infty(\Ltwo)}+ \|\ueps\|^2_{L^\infty(\Hone)} \\
	\leq&\, C_3\exp(C_4(1+ \|\phi\|_{X_\phi})T))(\|u_0\|^2_{\Hone}+\|u_1\|^2_{\Ltwo}+\|f\|^2_{L^1(\Ltwo)}),
	\end{aligned}
\end{equation}
for some $C_3$, $C_4>0$ that do not depend on $\eps$. We we will use this bound in the fixed-point proof to ensure the smallness of $\|\ueps\|_{L^\infty(\Ltwo)} \lesssim \|\ueps\|_{L^\infty(\Hone)}$.

\subsection{Fixed-point analysis} We proceed with the analysis of the nonlinear problem, where we will combine our previous results with a fixed-point argument.
\begin{theorem}\label{Thm:Wellp_2ndorder_nonlocal}
	Let ${k} \in \R$ and $\eps \in (0, \beps)$. Let $(u_0, u_1) \in \Honethree \times \Honetwo$ be such that
	\begin{equation}\label{initregWellp_2ndorder_nonlocal}
		\|u_0\|^2_{\Hthree}+\|u_1\|^2_{\Htwo}+\|f\|^2_{H^1(\Hone)} \leq r^2,
	\end{equation}
	where $r$ does not depend on $\varepsilon$. 
	Let assumptions \eqref{assumption1} and \eqref{assumption2} on the kernel $\frakKeps$ hold
	and let $f \in H^1(0,T; \Honezero)$. Then there exist data size
	$r_0=r_0 (r)>0$ and final time ${T}={T}(r)$, both independent of $\eps$, such that if  
	\begin{equation}\label{initsmallWellp_2ndorder_nonlocal}
		\|u_0\|^2_{\Hone}+\|u_1\|^2_{\Ltwo}+\|f\|^2_{L^1(\Ltwo)} \leq {r}_0^2,
	\end{equation}
	then there is a unique solution $\ueps \in  \calU$ of \eqref{ibvp_West_general}, with $\calU$ defined in \eqref{regularity}, which satisfies the following estimate uniformly in $\varepsilon$: 
	\begin{equation} \label{Nonlin:Main_energy_est}
		\begin{aligned}
		\|\ueps \|^2_{\calU} \leq C(T)\left( \|u_0\|_{\Hthree}^2+\|u_1\|^2_{\Htwo}+\|f\|^2_{H^1(\Hone)} \right).
		\end{aligned}
	\end{equation}  
\end{theorem}
\begin{proof}
	The proof follows by using the Banach fixed-point theorem on $\TK$, in the general spirit of~\cite[Theorem 6.1]{kaltenbacher2022parabolic} which considers uniform analysis in a local-in-time setting.  The mapping $\TK:\phi \mapsto \ueps$ is defined on the ball
\begin{equation}
	\begin{aligned}
		\phi \in \calB= \Big \{ \phi \in \calU:&\  \|\phi\|_{\calU} \leq R, \ 4|k|\|\phi\|_{L^\infty(\Linf)} \leq 1, 
		\, (\phi, \phi_t) \vert_{t=0} = (u_0, u_1) \Big \}
	\end{aligned}
\end{equation}
	 with a large enough radius $R>0$, which will be made precise below. The set $\mathcal{B} $ is non-empty as the solution of the linear problem with $k=0$ belongs to it if $R$ is sufficiently large, relative to $r$.  \\
\indent Let $\phi \in \mathcal{B}$.  Then the smoothness assumption on $\phi$ in Proposition~\ref{Prop:WellP_Lin} holds since $\calU \subset X_\phi$.
	Furthermore, since $4|k|\|\phi\|_{L^\infty(\Linf)} \leq 1$,  $\aaa$ does not degenerate:
	\[
	\begin{aligned}
		\frac12=	\underline{\aaa} \leq \aaa=1+2k\phi(x,t) \leq \overline{\aaa}=\frac32,  \quad (x,t)\in \Omega \times (0,T).
	\end{aligned}
	\]
	
	\emph{${\bf(1)}$ The self-mapping property.}  We wish to prove that $\ueps:=\TK \phi  \in \calB$. Since the assumptions of Proposition~\ref{Prop:WellP_Lin} hold, we know that $\ueps$ satisfies
	\begin{equation}
		\begin{aligned}
		\|u^\eps\|^2_{\calU} \leq	C_1 \exp( C_2(1+\|\phi\|_{X_\phi}+\ldots+\|\phi\|^6_{X_\phi}) T) (\|u_0\|^2_{\Hthree}+\|u_1\|^2_{\Htwo}+\|f\|^2_{H^1(\Hone)}).
		\end{aligned}
	\end{equation}
We can guarantee that $\|\ueps\|_{\calU} \leq R$ by choosing $R=R(r)$ and $T= T(r)$ so that
	\begin{equation}
		\begin{aligned}
		C_1 \exp( C_2(1+R+\ldots+R^6) T)	r^2 \leq R^2.
		\end{aligned}
	\end{equation}
	\indent We next show that the bound $4{|k|}\|\ueps\|_{L^\infty(\Linf)} \leq 1$  holds. To this end, we rely on Agmon's interpolation inequality \eqref{unif_bound_energynorm}
and combine it with the uniform lower-order bound in \eqref{unif_est_lowerorder}.  We have
\[
\|\ueps\|_{L^\infty(\Htwo)} \lesssim \|\ueps\|_{\calU} \lesssim R,
\]
and so using Agmon's inequality leads to
	\begin{align} \label{est_Agmon_p}
	&\|\ueps\|_{L^\infty(\Linf)} \\
		\leq&\, C_{\textup{A}}	\|\ueps\|^{1-d/4}_{L^\infty(\Ltwo)} 	\|\ueps\|^{d/4}_{L^\infty(\Htwo)} \\ 
		\leq&\, C_{\textup{A}}\left\{C_3\exp(C_4T(1+R))(\|u_0\|^2_{\Hone}+\|u_1\|^2_{\Ltwo}+\|f\|^2_{L^1(\Ltwo)})\right\}^{1/2-{d}/8}R^{{d}/4}.
	\end{align}
	We then choose data size ${r}_0=r_0(r)>0$ small enough, so that
	\[
	4|k| C_{\textup{A}}\left\{C_3\exp(C_4T(1+R)) r_0^2 \right\}^{1/2-{d}/8}R^{{d}/4}\leq 1.
	\]
Then $4|k| \|\ueps\|_{L^\infty(\Linf)}  \leq 1$, which was the last condition needed to conclude that $\ueps \in \calB$.\\

	\emph{${\bf (2)}$ Strict contractivity.} 
	 Take $\phi^{(1)}$ and $\phi^{(2)}$ in $\mathcal{B}$, and denote their difference by $\overline{\phi}= \phi^{(1)} -\phi^{(2)} $. Let $u^{\eps, (1)}=\TK (\phi^{(1)})$ and $u^{\eps, (2)}=\TK (\phi^{(2)})$. Their difference $\overline{u}=u^{\eps, (1)} -u^{\eps, (2)} \in \mathcal{B}$ then solves 
	\begin{equation} \label{2ndorder_nonlocal_diff_contr}
		\begin{aligned}
			(1+ 2{k} \phi^{(1)})\overline{u}_{tt}-c^2 \Delta \overline{u}-    \frakKeps*\D \overline{u}_t+2k \phi^{(1)}_t \bar{u}_t 
			=\, \begin{multlined}[t] -2k \overline{\phi} u_{tt}^{\eps, (2)} -2k \overline{\phi}_t u_t^{\eps, (2)} \end{multlined}
		\end{aligned}
	\end{equation}
	with zero initial and boundary data. Since the right-hand side of the above equation does not belong to $H^1(0,T; \Hone)$, we cannot use estimate \eqref{final_est_lin} to prove contractivity of the mapping  $\mathcal{T}$ with respect to $\|\cdot\|_{\calU}$. Instead, we prove that the mapping $\mathcal{T}$ is strictly contractive in the space
		\begin{equation} \label{def_space_XE}
			X_{\textup{E}}:=	W^{1,\infty}(0,T; \Ltwo) \cap L^\infty(0,T; \Honezero),
		\end{equation}
		endowed with the norm defined in \eqref{energy_norm}. We can again use the uniform lower-order bound established in \eqref{unif_est_lowerorder} to conclude that
	\begin{equation}
		\begin{aligned}
		&\|\overline{u}_t\|_{L^\infty(\Ltwo)}+ \|\overline{u}\|_{L^\infty(\Hone)} \\
		\lesssim&\, \|-2k \overline{\phi} u_{tt}^{\eps, (2)} -2k \overline{\phi}_t u_t^{\eps, (2)}\|_{L^1(\Ltwo)}\\
		\lesssim&\,|k| \|\overline{\phi}\|_{L^2(\Lfour)}\|u_{tt}^{\eps, (2)} \|_{L^2(\Lfour)} + |k|\|\overline{\phi}_t\|_{L^2(\Ltwo)}\|u_t^{\eps, (2)}\|_{L^2(\Linf)}.
		\end{aligned}
	\end{equation}
	By using the embeddings $\Hone \hookrightarrow \Lfour$ and $\Htwo \hookrightarrow \Linf$ and the fact that $u^{\eps, (2)} \in \calB$, we further infer that
	\begin{equation}
		\begin{aligned}
		\|\overline{u}_t\|_{L^\infty(\Ltwo)}+ \|\overline{u}\|_{L^\infty(\Hone)} 
		 \lesssim&\, |k| (\|\overline{\phi}\|_{L^2(\Hone)}+\|\overline{\phi}_t\|_{L^2(\Ltwo)})\\
		 \lesssim&\,  |k|\sqrt{T}(\|\overline{\phi}\|_{L^\infty(\Hone)}+\|\overline{\phi}_t\|_{L^\infty(\Ltwo)}).
		\end{aligned}
	\end{equation}
	Therefore, by reducing ${{T}}$ if needed (independently of $\eps$), 
	we can ensure that the mapping $\mathcal{T}$ is strictly contractive with respect to the energy norm \eqref{energy_norm}. \\
	\indent The arguments showing that $\calB$ is closed with respect to this norm are analogous to those of \cite[Theorem 4.1]{kaltenbacher2022parabolic}. By the Banach fixed-point theorem, we therefore obtain a unique fixed point $\ueps=\mathcal{T}(\ueps)$ in $\calB$, which solves the nonlinear problem. 
\end{proof}

We note that as a byproduct of Theorem~\ref{Thm:Wellp_2ndorder_nonlocal}, we obtain local well-posedness for the strongly damped ($\frakKeps=\eps \delta_0$), inviscid ($\frakKeps \equiv 0$), and fractionally damped ($\frakKeps=\eps g_\alpha$) Westervelt equations, therefore generalizing the results of~\cite{dorfler2016local,kaltenbacher2022inverse,kaltenbacher2022parabolic}.

%% file: Limit_delta_2ndorder.tex
\section{Establishing the limiting behavior for relevant classes of kernels} \label{Sec:LimitingAnalysis}
We next wish to determine the limiting behavior of solutions to \eqref{West_general} as $\eps \searrow 0$. 
We begin by proving (H\"older) continuity of the solution with respect to the memory kernel. 
\begin{theorem} \label{Thm:Limit}
Let $\eps_1, \eps_2 \in (0, \beps)$. Under the assumptions Theorem~\ref{Thm:Wellp_2ndorder_nonlocal}, the following estimates hold: 
\begin{equation} \label{continuity_kernels}
	\|u^{\eps_1}-u^{\eps_2}\|_{\textup{E}}\lesssim 
			\|(\frakK_{\eps_1}-\frakK_{\eps_2})\Lconv1\|^{1/2}_{L^1(0,T)},
\end{equation}
and
\begin{equation} \label{continuity_kernels_L1}
	\|u^{\eps_1}-u^{\eps_2}\|_{\textup{E}}\lesssim  \|\frakK_{\eps_1}-\frakK_{\eps_2}\|_{\calM(0,T)}. 
\end{equation}

\end{theorem}
\begin{proof}
To prove the statement, we can see the difference $\overline{u}=u^{\eps_1}-u^{\eps_2}$ as the solution to 
	\begin{equation} \label{deltaLim_diff_eq}
		\begin{aligned}
			&\begin{multlined}[t]
				((1+ 2{k} u^{\eps_1} )\overline{u}_t)_t-c^2 \Delta \overline{u}-  \frakK_{\eps_1}*\D \overline{u}_t + 2k (\overline{u} u_t^{\eps_2})_t \end{multlined}
			= \,  (\frakK_{\eps_1}-\frakK_{\eps_2})* \Delta u_t^{\eps_2}
		\end{aligned}
	\end{equation}
with zero initial and boundary conditions. We can test this equation with $\overline{u}_t$ and proceed similarly to the proof of contractivity in Theorem~\ref{Thm:Wellp_2ndorder_nonlocal} with now
	$
	\aaa=1+2{k}  u^{\eps_1}$. The new term compared to before is the convolution on the right-hand side. After testing, this term can be handled using Young's convolution inequality as follows:
	\begin{equation} \label{diff_kernel_L1}
	\begin{aligned}
&\inttO	(\frakK_{\eps_1}-\frakK_{\eps_2})* \Delta u_t^{\eps_2}\overline{u}_t \dxs \\
\lesssim&\, \|\frakK_{\eps_1}-\frakK_{\eps_2}\|_{\calM(0,T)} \|\Delta u_t^{\eps_2}\|_{L^2(\Ltwo)}\|\overline{u}_t\|_{L^2(\Ltwo)}\\
 \lesssim&\, \|\frakK_{\eps_1}-\frakK_{\eps_2}\|^2_{\calM(0,T)} \|\Delta u_t^{\eps_2}\|^2_{L^2(\Ltwo)}+\|\overline{u}_t\|^2_{L^2(\Ltwo)}.
\end{aligned}
	\end{equation}
	Since by Theorem~\ref{Thm:Wellp_2ndorder_nonlocal}, $\|\Delta u_t^{\eps_2}\|_{L^2(\Ltwo)} \lesssim \|u^{\eps_2}\|_\calU$ is uniformly bounded, we obtain the claimed estimate by relying on Gr\"onwall's inequality. \\
	\indent The difference of kernels term can also be treated using integration by parts in space:
	\begin{equation}\label{psitaudiff}
		\begin{aligned}
			&\int_0^t (((\mathfrak{K}_{\eps_1}-\mathfrak{K}_{\eps_2})* \Delta \ut^{\eps_2})(s), \overline{u}_t(s))_{L^2}\ds \\
				=&\, \int_0^t \Big \{ ((1*(\mathfrak{K}_{\eps_1}-\mathfrak{K}_{\eps_2})*\Delta u_{tt}^{\eps_2})(s), \overline{u}_t(s))_{L^2}+((1*(\mathfrak{K}_{\eps_1}-\mathfrak{K}_{\eps_2}))(s)\Delta u_1, \overline{u}_t(s))_{L^2} \Big\} \ds \\
			=&\, \int_0^t \Big \{-((1*(\mathfrak{K}_{\eps_1}-\mathfrak{K}_{\eps_2})*\nabla u_{tt}^{\eps_2})(s), \nabla \overline{u}_t(s))_{L^2}+((1*(\mathfrak{K}_{\eps_1}-\mathfrak{K}_{\eps_2}))(s)\Delta u_1, \overline{u}_t(s))_{L^2} \Big\}\ds .
			\end{aligned}
	\end{equation}
	Then by Young's convolution inequality
		\begin{equation}
	\begin{aligned}
	&\int_0^t (((\mathfrak{K}_{\eps_1}-\mathfrak{K}_{\eps_2})* \Delta \ut^{\eps_2})(s), \overline{u}_t(s))_{L^2}\ds \\
	\leq&\,\begin{multlined}[t]  \|(\mathfrak{K}_{\eps_1}-\mathfrak{K}_{\eps_2})*1\|_{L^1(0,T)} \left(\|\nabla \utt^{\eps_2}\|_{L^1_t(\Ltwo)}\|\nabla \overline{u}_t\|_{L^\infty_t(\Ltwo)}+
	\|\Delta u_1\|_{\Ltwo}\|\overline{u}_t\|_{L^1_t(\Ltwo)} \right).\end{multlined}
	\end{aligned}
	\end{equation}
	Note that we would not be able to absorb the term $\|\nabla \overline{u}_t\|^2_{L^\infty_t(\Ltwo)}$ by the left-hand side (i.e., the energy norm). However, by  Theorem~\ref{Thm:Wellp_2ndorder_nonlocal}, we know that the following uniform bound holds:
	\[
	\|\nabla \utt^{\eps_2}\|_{L^1_t(\Ltwo)}\|\nabla \overline{u}_t\|_{L^\infty_t(\Ltwo)}+
	\|\Delta u_1\|_{\Ltwo}\|\overline{u}_t\|_{L^1_t(\Ltwo)}  \leq C.
	\]
	Proceeding otherwise as in the proof of contractivity in Theorem~\ref{Thm:Wellp_2ndorder_nonlocal}, we therefore obtain
	\begin{equation}\label{psitaudiff_E}
		\|u^{\eps_1}-u^{\eps_2}\|^2_{\textup{E}}\leq C  \|(\mathfrak{K}_{\eps_1}-\mathfrak{K}_{\eps_2})*1\|_{L^1(0,T)}
	\end{equation}
	with a constant $C>0$ that is independent of $\eps_{1,2}$, from which the claim follows.
\end{proof}
Observe that the main culprit for the reduced order of continuity in \eqref{continuity_kernels} (i.e., having $\|(\frakK_{\eps_1}-\frakK_{\eps_2})\Lconv1\|^{1/2}_{L^1(0,T)}$ instead of $\|(\frakK_{\eps_1}-\frakK_{\eps_2})\Lconv1\|_{L^1(0,T)}$) is the integration by parts with respect to space in \eqref{psitaudiff}. This approach was forced by a lack of uniform bound on $\Delta \utt^{\eps_2}$. We thus expect that the order can be improved in a more regular setting in terms of data that would lead to a uniform estimate on $\Delta \utt^{\eps}$ in \eqref{ibvp_West_general}. \\
\indent Theorem~\ref{Thm:Limit} is the key to establishing the limiting behavior of solutions to \eqref{ibvp_West_general} as $\eps \searrow 0$ and, in particular, the convergence rates. As they will depend on the form of the kernel $\frakKeps$ and, in turn, its dependence on $\varepsilon$, we treat different classes of kernels separately. 
\subsection{The vanishing sound diffusivity limit with fractional-type kernels}  \label{Sec:Limit_delta}
We first discuss the setting $\frakKeps=\eps \frakK$. Recall that a representative example of this class of kernels (up to a constant) is 
\begin{equation}
	\frakKeps=\eps   g_{\alpha},
\end{equation}
where the Abel kernel is defined in \eqref{def_galpha} for $\alpha \in (0,1)$ and $g_0= \delta_0$. We will prove in Section~\ref{sec:verification} that this kernel indeed verifies assumption \eqref{assumption2}. 
\begin{corollary}\label{Corollary:Limit_epsK}
	Under the assumptions of Theorem~\ref{Thm:Wellp_2ndorder_nonlocal} with the kernel
	\[
	\frakKeps = \eps \frakK, \quad \eps \in (0, \beps)
	\] 
	satisfying assumptions \eqref{assumption1} and \eqref{assumption2}, the family of solutions $\{u^{\eps}\}_{\eps \in (0, \beps)}$ of \eqref{ibvp_West_general}  converges in the energy norm to the solution $u \equiv u^0$ of the initial boundary-value problem for the inviscid Westervelt equation
	\begin{equation}\label{ibvp_West_limit}
		\left \{	\begin{aligned} 
			&((1+2ku)\ut)_t-c^2 \Delta u  = f \quad  &&\text{in } \Omega \times (0,T), \\
			&u =0 \quad  &&\text{on } \partial \Omega \times (0,T),\\
			&(u, \ut)=(u_0, u_1), \quad  &&\text{in }  \Omega \times \{0\},
		\end{aligned} \right.
	\end{equation}
 at a linear rate
 \[
 \|u-\ueps\|_{\textup{E}} \lesssim \eps.
 \]
\end{corollary}
\begin{proof}
In this setting, the limiting kernel  is $\frakK_0=0$, and it satisfies assumptions \eqref{assumption1} and \eqref{assumption2}. By Theorem~\ref{Thm:Limit} and estimate \eqref{continuity_kernels_L1}, we then immediately have
\begin{equation}
	\begin{aligned}
		\|\ueps-u\|_{\textup{E}} \leq C \eps \|\frakK\|_{\calM(0,T)} 
	\end{aligned}
\end{equation}
for some $C>0$, independent of $\eps$, from which the statement follows.
\end{proof}
 The limiting result above largely generalizes the one of~\cite[Theorem 4.1]{kaltenbacher2022parabolic}, where $\frakK$ is the Dirac delta distribution $\delta_0$. Here we allow for a large class of memory kernels, thus establishing the vanishing sound diffusivity limit for the Westervelt equation with general dissipation of time-fractional type.

\subsection{The vanishing thermal relaxation time limit with Mittag-Leffler kernels}  
We now turn our attention to the kernels that were motivated by the presence of thermal relaxation in the heat flux laws of the propagation medium, and have the form
\begin{equation} \label{ML_kernel_Sec4}
	\begin{aligned}
		\frakKeps(t)=&\,\left(\frac{\rt}{\eps}\right)^{a-b}\frac{1}{\eps^b}t^{b-1}E_{a,b}\left(-\left(\frac{t}{\eps}\right)^a\right), \quad a,b \in (0,1].
	\end{aligned}
\end{equation}
Before continuing with the singular limit analysis, it is helpful to recall certain properties of the Mittag-Leffler functions, which can be found, for example, in~\cite{GorenfloMainardiRogosin, jin2021fractional}. \subsection*{Properties of the Mittag-Leffler functions}  We recall that the functions
\begin{equation}
t\mapsto E_{a, b}\left(-\lambda t^a\right), \ a \in [0,1], \ t>0, \ \lambda>0
\end{equation}
are completely monotone for  $b \geq a$ (and, in particular, non-negative); see~\cite[Corollary 3.2]{jin2021fractional}.  Furthermore, the following identity holds:
\begin{equation} \label{identity_ML}
t^{b-1}E_{a,b}(-t^a)=\ddt\Bigl(t^{b}E_{a,b+1}(-t^a)\Bigr).
\end{equation}
We also recall that the Laplace transform of the Mittag-Leffler functions is given by
\begin{equation} \label{Laplace_formula_ML}
\mathscr{L}[t^{b-1}E_{a,b}(- \lambda t^a)](z) =\frac{z^{a-b}}{z^a+\lambda}, \quad a,b>0,\  \Re(z)>0, \  \lambda \geq 0; 
\end{equation}
see~\cite[Lemma 3.2]{jin2021fractional}.  We will additionally rely on the asymptotic behavior of Mittag-Leffler functions:
\begin{align}  \label{asymptotics}
E_{a,b}(-x)\sim\frac{1}{\Gamma(b-a)\, x} {\mbox{ as }x\to\infty}; 
\end{align}
 see, e.g., \cite[Theorem 3.2]{jin2021fractional} .\\ 
\indent In what follows, we intend to take the limit $\eps \searrow0$, while keeping $\rt>0$ fixed. We treat the cases $a-b \leq 0$ and $a-b>0$ (where additionally $\rt/\eps>0$ should be fixed) separately when discussing the limiting behavior.

\subsubsection{Limiting behavior for $a-b \leq 0$}\label{sec:lim_tau_fixedtautheta}

If $0< a\leq b \leq 1$, we will prove that solutions $\ueps$ of \eqref{ibvp_West_general} converge to the solution $u$ of the following time-fractional equation: 
\begin{equation} \label{ibvp_2ndorder_tau0}
	\begin{aligned}
&	((1+2{k} u)\ut)_t-c^2 \Delta u -  \rt^{a-b}
{\textup{D}_t^{a-b+1}\D u} =f,	\end{aligned} 
\end{equation}
supplemented by the same boundary and initial conditions as in \eqref{ibvp_West_general}. Recall that 
$$\textup{D}_t^{a-b+1}\D u = g_{b-a} \Lconv \D u_t. $$ Note also that in case $a=b$, the limiting equation is strongly damped. 
\begin{proposition}\label{Prop:Limit_a<=b} 
Let $\rt>0$ be fixed. Under the assumptions of Theorem~\ref{Thm:Wellp_2ndorder_nonlocal}, the family of solutions $\{\ueps\}_{\eps \in(0,\beps)}$ of  \eqref{ibvp_West_general} with the kernel given by
\begin{equation} \label{kernel_form}
	\begin{aligned}
		\frakKeps(t)=&\,\left(\frac{\rt}{\eps}\right)^{a-b}\frac{1}{\eps^b}t^{b-1}E_{a,b}\left(-\left(\frac{t}{\eps}\right)^a\right) \quad \text{where } a-b \leq 0, \ \, a, b \in (0, 1],
	\end{aligned}
\end{equation}
converges to the solution $u$ of 
\begin{equation}\label{ibvp_West_fractional_limit_tau}
	\left \{	\begin{aligned} 
		&((1+2ku)\ut)_t-c^2 \Delta u -  \frakK_0 *\Delta \ut = f \quad  &&\text{in } \Omega \times (0,T), \\
		&u =0 \quad  &&\text{on } \partial \Omega \times (0,T),\\
		&(u, \ut)=(u_0, u_1), \quad  &&\text{in }  \Omega \times \{0\},
	\end{aligned} \right.
\end{equation}
with the kernel $\frakK_0= \rt^{a-b} g_{b-a}$ in the following sense:
	\begin{equation} \label{linrate_Kuzn_tau}
	\|\ueps-u\|_{\textup{E}}\lesssim  \|(\frakKeps-\frakK_0)\Lconv1\|^{1/2}_{L^1(0,T)}  \sim\eps^{a/2}
	\quad \mbox{ as } \ \eps  \searrow 0.
	\end{equation}
\end{proposition}
\begin{proof}
By Theorem~\ref{Thm:Limit} (up to checking that $\frakK_0$ satisfies \eqref{assumption1} and \eqref{assumption2}), we have
\begin{equation}
\| \ueps-u\|_{\textup{E}}\leq C  \|(\mathfrak{K}_{\eps}-\mathfrak{K}_{0})*1\|^{1/2}_{L^1(0,T)}
\end{equation}
with a constant $C>0$ that is independent of $\eps$. We next rely on the Laplace transform to further establish the asymptotic behavior of the right-hand side as $\eps \searrow 0$. To this end, we use \eqref{Laplace_formula_ML} and the formula
\[
\mathscr{L}[\frakK_0](z)= \rt^{a-b} \mathscr{L}[g_{b-a}](z)= \rt^{a-b} z^{a-b}.
\]
\indent Without loss of generality, we assume $\rt =1$. Consider the Laplace transform of $(\frakKeps-\frakK_0)*1$:
\begin{align}
\mathscr{L}[(\frakKeps-\frakK_0)*1](z)=&\, (\mathscr{L}[\frakK_{\eps}]- \mathscr{L}[\frakK_0])(z)\mathscr{L}[1](z) \\
=&\, \left(\frac{z^{a-b}}{(\eps z)^a+1}- z^{a-b}\right)\frac{1}{z}
=-\eps^a \frac{z^{a-(1+b-a)}}{(\eps z)^a+1}.
\end{align}
From here we conclude that \[((\frakKeps-\frakK_0)*1)(t)=-t^{b-a}E_{a,1+b-a}\left(-\left(\frac{t}{\eps}\right)^a\right).\] On account of the non-negativity of the function $t\mapsto E_{a,1+b-a}\left(-\left(\frac{t}{\eps}\right)^a\right)$ (which is ensured by $1+b-a \geq a$ due to the assumptions on $a$ and $b$), we find that
\[
\begin{aligned}
\|(\frakKeps-\frakK_0)*1\|_{L^1(0,T)}
=&\, \int_0^T t^{b-a}E_{a,1+b-a}\left(-\left(\frac{t}{\eps}\right)^a\right) \dt \\
=&\, T^{1+b-a}E_{a,2+b-a}\left(-\left(\frac{T}{\eps}\right)^a\right).
\end{aligned}
\] 
Then by asymptotic properties of the Mittag-Leffler functions in \eqref{asymptotics}, we have
\[
\begin{aligned}
	\|(\frakKeps-\frakK_0)*1\|_{L^1(0,T)}=&\, T^{1+b-a}E_{a,2+b-a}\left(-\left(\frac{T}{\eps}\right)^a\right)\\
	\sim&\, \frac{T^{1+b-a}}{\Gamma(2+b-2a)} \left(\frac{T}{\eps}\right)^{-a} \qquad
	\text{ as }\ \frac{T}{\eps}\to\infty,
\end{aligned}
\] 
which concludes the proof.
\end{proof}
\subsubsection{Limiting behavior for $a>b$}\label{sec:lim_tau_fixedratio}
We next discuss the limiting behavior of solutions to \eqref{ibvp_West_general} in the remaining case for the Mittag-Leffler kernels, which is $1 \geq a>b>0$. To take the limit $\eps\searrow0$ of \eqref{ibvp_West_general} with the kernel \eqref{ML_kernel_Sec4} when $a>b$, now we need the following additional assumption on $\rt$:
	\begin{equation}\label{assumptiion_rt/eps_const}
		\rt = \rt(\eps) \quad \textrm{and } \left(\frac{\rt}{\eps}\right)^{a-b} = \rho^{a-b} = \textrm{constant},
	\end{equation}
under which the kernels have the form
\begin{equation}
	\begin{aligned}
		\frakKeps(t)=&\,\rho^{a-b} \frac{1}{\eps}\left(\frac{t}{\eps} \right)^{b-1}E_{a,b}\left(-\left(\frac{t}{\eps}\right)^a\right), \quad 1 \geq a>b>0.
	\end{aligned}
\end{equation}
\begin{proposition}\label{Prop:Limit_a>b}
Under the assumptions of Theorem~\ref{Thm:Wellp_2ndorder_nonlocal} and assumption \eqref{assumptiion_rt/eps_const}, 
the family of solutions $\{\ueps\}_{\eps \in(0,\beps)}$ of \eqref{ibvp_West_general} with
\begin{equation}
	\begin{aligned}
		\frakKeps(t)=&\,\left(\frac{\rt}{\eps}\right)^{a-b}\frac{1}{\eps^b}t^{b-1}E_{a,b}\left(-\left(\frac{t}{\eps}\right)^a\right), \qquad 1 \geq a>b>0,
	\end{aligned}
\end{equation}
converges in the energy norm to the solution $u$ of the inviscid problem \eqref{ibvp_West_limit} at the following rate:
	\begin{equation} \label{linrate_Kuzn_tau2}
	\|\ueps-u\|_{\textup{E}}\lesssim \|\frakKeps*1\|^{1/2}_{L^1(0,T)} \sim \eps ^{(a-b)/2} \quad \mbox{ as } \ \eps\searrow0.
	\end{equation} 
\end{proposition}

The statement of Proposition~\ref{Prop:Limit_a>b} can seem at first inspection unintuitive. Indeed, with the ratio $\rt/\eps$ constant as assumed in \eqref{assumptiion_rt/eps_const}, we have
	 \[
	\frakKeps(t) = \rho^{a-b }\frac{1}{\eps} \frakK\left(\frac{t}{\eps}\right),
	\] 
	where $\frakK$ does not depend on $\eps$. One would naively expect that as the rescaling parameter $\eps \searrow 0$, the kernel would converge in the sense of distributions to a Dirac mass at 0. Results exist in this sense; see \cite{conti2005singular,conti2006singular}.
	However, in our setting, we do not require $\frakK \geq 0$ which allows for the kernel $\frakKeps$ to converge to the zero function even if $\frakK \neq 0$.
\begin{proof}
The statement follows by Theorem~\ref{Thm:Limit} and the asymptotic behavior of $\|(\frakKeps-\frakK_0)*1\|_{L^1(0,T)}=\|\frakKeps*1\|_{L^1(0,T)}$ as $\eps \searrow 0$. To establish the latter, we note that
	\[
	\mathscr{L}[\frakKeps](z)=\rho^{a-b}\frac{(\eps z)^{a-b}}{(\eps z)^a+1}\to 0\quad  \text{as} \ \eps \searrow 0
	\]	
	since $a>b$. Considering the identity 
	\[
	\mathscr{L}[\frakKeps*1](z)=\rho^{a-b}\frac{1}{z}\frac{(\eps z)^{a-b}}{(\eps z)^a+1}
=\rho^{a-b}\eps \frac{(\eps z)^{a-(1+b)}}{(\eps z)^a+1},
	\]
we find, using the scaling property of the Laplace transform, that
	\[	(\frakKeps*1)(t)=\rho^{a-b}\left(\frac{t}{\eps}\right)^{b}E_{a,b+1}\left(-\left(\frac{t}{\eps}\right)^a\right).
	\] 
	Now, using a change of variable, we have
	\[
	\begin{aligned}
		\|\frakKeps*1\|_{L^1(0,T)}=\eps  \rho^{a-b} \int_0^{T/\eps} \vert t^{b}E_{a,b+1}(-t^a) \vert \dt.
	\end{aligned}
	\] 
Since $b+1\geq a$, the Mittag-Leffler function is completely monotone and thus non-negative. Furthermore, identity  \eqref{identity_ML} and asymptotics \eqref{asymptotics} imply that
\begin{equation} \label{asymptotics_a_geq_b}
\begin{aligned}
	\eps \rho^{a-b} \int_0^{T/\eps} t^{b}E_{a,b+1}(-t^a) \dt  
=&\, \eps \rho^{a-b}\left(\frac{T}{\eps}\right)^{b+1}E_{a,b+2}\left(-\left(\frac{T}{\eps}\right)^a\right) \\
\sim&\, \rho^{a-b} \frac{T^{1+b-a}}{\Gamma(b+2-a)} \eps^{a-b},
\end{aligned}
\end{equation}
which completes the proof.
\end{proof}
\begin{remark}
Assumption \eqref{assumptiion_rt/eps_const} means that $\rt$ should converge to zero at least as fast as $\eps$. Since \eqref{asymptotics_a_geq_b} implies that
\begin{equation} 
	\begin{aligned}
		\eps \rho^{a-b} \int_0^{T/\eps} t^{b}E_{a,b+1}(-t^a) \dt  
		\sim&\, \rt^{a-b} \frac{T^{1+b-a}}{\Gamma(b+2-a)},
	\end{aligned}
\end{equation}
one could remove this assumption by letting $\rt \searrow 0$ and keeping $\eps>0$ (that is, the thermal relaxation time) fixed.
\end{remark}

%% file: kernelverification.tex
\section{Verifying the uniform boundedness and coercivity assumptions for different classes of kernels} \label{sec:verification}
The remainder of this paper is devoted to proving that the different classes of kernels discussed in Section~\ref{Sec:Gurtinpipkin_laws} and the limiting kernels $\frakK_0$ from the previous section indeed satisfy assumptions \eqref{assumption1} and \eqref{assumption2} imposed in the well-posedness and singular limit analysis. 

\subsection{How to verify assumption \eqref{assumption1}} For kernels that have the form $\frakKeps= \eps \frakK$ with $\eps \in (0, \beps)$ and $\frakK \in L^1(0,T) \cup \{\delta_0\}$ independent of $\eps$, we immediately obtain
\[
\|\frakKeps\|_{\calM(0,T)} \leq \beps \|\frakK\|_{\calM(0,T)}.
\]
\indent The Mittag-Leffler kernels are more interesting from the point of view of obtaining the uniform $L^1(0,T)$ bound. 
Since the kernels in the Gurtin--Pipkin forms of the Compte--Metzler laws can be written as
\begin{equation} \label{ML_kernels_eps}
\frakKeps(t)= \left(\frac{\rt}{\eps}\right)^{a-b} \frac{1}{\eps} \frakK\left(\frac{t}{\eps}\right), 
\end{equation}
where
\begin{equation} \label{def_ML_kernel_frakK}
	\frakK(t) = t^{b-1} E_{a,b}(-t^a),
\end{equation}
 we have
\begin{equation} \label{L1_frakKeps_frakK}
\|\frakKeps\|_{L^1(0,T)}= \left(\frac{\rt}{\eps}\right)^{a-b}\|\frakK \|_{L^1(0,T/\eps)}.
\end{equation}
As in the limiting analysis, we discuss different cases with respect to $a$ and $b$ to determine the behavior of the right-hand side term as $\eps\searrow0$.\\

\noindent $\bullet$ 
If $a=b$, then
\[
\|\frakKeps\|_{L^1(0,T)}=  \|\frakK \|_{L^1(0,T/\eps)}.
\]
Therefore, as $\eps \searrow 0 $, we need to require that
$\frakK \in L^1(0,\infty)$. Recall that in this case the Mittag-Leffler function is completely monotone and therefore, in particular, non-negative. Furthermore,  identity \eqref{identity_ML} implies
\[
\|\frakK\|_{L^1(0,r)}
=\int_0^r \ddt\Bigl(t^{b}E_{a,b+1}(-t^a)\Bigr)\dt = r^{b}E_{a,b+1}(-r^a). 
\]
By the asymptotic behavior of Mittag-Leffler functions \eqref{asymptotics}, we find that
\begin{equation}\label{eq:uniformL1_GPLII}
	t^{b}E_{a,b+1}(-t^a)\sim\frac{1}{\Gamma(b+1-a)} t^{b-a} \quad \text{ as } \ t\to\infty.
\end{equation}
This asymptotic behavior implies that $\|\frakK\|_{L^1(0,\infty)}$ with $\frakK$ given in \eqref{def_ML_kernel_frakK} is finite with $a=b$. \\
\indent Among the kernels in Section~\ref{Sec:Gurtinpipkin_laws} motivated by the physics of nonlinear acoustics, this assumption holds true for the exponential kernel \eqref{kernel_exp} where $(a,b)=(1,1)$ and the kernel coming from the GFE law where $(a,b)=(\alpha, \alpha)$; see Table~\ref{tab:ker_par_mod}.
\\

\noindent $\bullet$ If $a<b$, asymptotics \eqref{eq:uniformL1_GPLII} is valid but the function tends to infinity.   
However, note that for $\rt>0$ fixed, we have a uniform bound on $\|\frakKeps\|_{L^1(0,T)}$ on account of
\[
\|\frakKeps\|_{L^1(0,T)}=\left(\frac{\rt}{\eps}\right)^{a-b} \|\frakK\|_{L^1(0,T/\eps)} \sim \frac{1}{\Gamma(b+1-a)} \left(\frac{T}{\rt}\right)^{b-a} \quad \textrm{as } \eps \searrow 0.
\]
Among the kernels discussed in Section~\ref{Sec:Gurtinpipkin_laws}, the assumption $a<b$ holds for the GFE II law, where $(a,b)=(\alpha,1)$; see Table~\ref{tab:ker_par_mod}. 
\\

\noindent $\bullet$ If $a>b$, then in view of \eqref{L1_frakKeps_frakK}, inevitably
\[\|\frakKeps\|_{L^1(0,T)}\to +\infty \ \text{ as } \eps\searrow0, \] except in the trivial case $\frakK \equiv 0$. This asymptotic behavior provides a motivation for assuming the ratio $\rt/\eps$ to be fixed in the limiting analysis when $a>b$ in \eqref{ML_kernels_eps}, so that we have
\begin{equation}
\|\frakKeps\|_{L^1(0,T)}= \rho^{a-b}\|\frakK \|_{L^1(0,T/\eps)}.
\end{equation}	The physical interpretation of this assumption is that one needs a scaling parameter, $\rt$, to match $\tau$ on the right-hand side of the fractional flux laws  GFE I and GFE III; see Table~\ref{tab:ker_par_mod}. Thus if assumption \eqref{assumptiion_rt/eps_const} holds, we can make direct use of asymptotics \eqref{asymptotics}, from which we conclude 
\[\frakK(t) \sim \frac{1}{\Gamma(b-a)} t^{b-1-a} \ \text{ as } \ t\to\infty.\] Therefore, $\|\frakK\|_{L^1(0,\infty)}<\infty$.

\subsection{How to verify assumption \eqref{assumption2}} There are different ways of verifying coercivity assumption \eqref{assumption2}; we discuss two of them here  needed to tackle different classes of kernels we have seen so far. \\

$\bullet$ {\bf Fourier approach to proving coercivity.} 
One approach of verifying \eqref{assumption2}, which will be used for Mittag-Leffler kernels when $a\geq b$, is to employ Fourier analysis, similarly to \cite[Lemma 2.3]{Eggermont1987}. To this end, let us denote by $\overline{f}^t$ the extension of a function $f$ by zero outside $\mathbb{R}\setminus(0,t)$, and by $\mathcal{F}$, the Fourier transform. The following
identity holds for the Mittag-Leffler functions:
\[
\mathcal{F}[\overline{x^{b-1}E_{a,b}(\lambda x^a)}^\infty](\omega) =\frac{1}{\sqrt{2\pi}}\frac{(\imath\omega)^{a-b}}{(i\omega)^a-\lambda}, \quad \lambda \in \R.
\]
Using this formula, we can obtain the Fourier transforms of the Mittag-Leffler kernels in Section~\ref{Sec:Gurtinpipkin_laws}; see Table~\ref{table:kernels}. Let us also denote by $\Fconv$ the Fourier convolution \[(f\Fconv g)(t)=\int_{\mathbb{R}} f(t-s)g(s)\ds,\] and by an overline the complex conjugation. Furthermore, let 
\[ v(\omega)=(\mathcal{F}\overline{\frakKeps}^\infty)(\omega) (\mathcal{F}\overline{y}^t)(\omega), \qquad m(\omega)= \sqrt{2\pi}/(\mathcal{F}\overline{\frakKeps}^\infty)(\omega).\]
\begin{table}[h]	
	\captionsetup{width=13cm}
	\begin{adjustbox}{max width=\textwidth}
		\begin{tabular}[h]{|m{5cm}||m{3.5cm}|m{5.2cm}|}
			\hline
			\vspace*{1mm} \hspace*{1.8cm}$\frakKeps(t)$& 	\vspace*{1mm}$\sqrt{2\pi}\,\mathcal{F}\overline{\frakKeps}^\infty(\omega)$ 	\vspace*{1mm}&$\frac{1}{2\pi}m(\omega):=1/(\sqrt{2\pi}\mathcal{F}\overline{\frakKeps}^\infty(\omega))$ 
			\\
			\hline\hline
			\vspace*{1mm}	$	\left(\dfrac{\rt}{\eps}\right)^{a-b}\dfrac{1}{\eps^b}t^{b-1}E_{a,b}\left(-\left(\frac{t}{\eps}\right)^a\right)$ & ${\left(\dfrac{\rt}{\eps}\right)^{a-b}\dfrac{(\eps\imath\omega)^{a-b}}{(\eps\imath\omega)^a+1}}$&  ${\left(\dfrac{\rt}{\eps}\right)^{b-a} \left[(\eps\imath\omega)^{b}+(\eps\imath\omega)^{b-a}\right]}$\\	
			\hline
		\end{tabular}
	\end{adjustbox}
	~\\[2mm]
	\caption{\small  Fourier transforms of the Mittag-Leffler kernels discussed in Section~\ref{Sec:Gurtinpipkin_laws}} \label{table:kernels}
\end{table}
~\\
Then using Plancherel's theorem and the Fourier convolution theorem (see, e.g. \cite[Theorems 1,2, Section 4.3]{evans2010partial}), we have
\begin{equation}\label{coercivity_calK_RecalFcalK}
	\begin{aligned}
		\int_0^t\Bigl((\frakKeps\Lconv y)(s),y(s)\Bigr)_{L^2}\ds 
		=&\, \int_\mathbb{R} \left((\overline{\frakKeps}^\infty\Fconv\overline{y}^t)(s), \overline{y}^t(s)\right)_{L^2}\ds\\
		=&\, \sqrt{2\pi} \int_\mathbb{R} \left((\mathcal{F}\overline{\frakKeps}^\infty)(\omega) (\mathcal{F}\overline{y}^t)(\omega), \overline{(\mathcal{F}\overline{y}^t)(\omega)}\right)_{L^2}  \domega \\
		=&\, \int_\mathbb{R} \overline{m(\omega)} \left\|v(\omega)\right\|_{\Ltwo}^2  \domega \\
		=&\, \int_\mathbb{R} \Re(m(\omega)) \left\|v(\omega)\right\|_{\Ltwo}^2  \domega, 
	\end{aligned}
\end{equation}
where we have also used the fact that the left-hand side of the identity is real valued.
Thus, since the real part of the Fourier transform of a real-valued function is even, if
\begin{equation}\label{ReFcalK}
	\Re(m(\omega))=\Re(\sqrt{2\pi}/\mathcal{F}\overline{\frakKeps}^\infty)(\omega)\geq 2\pi \,\tilde{C}_{a,b}, \quad  \omega\in (0,\infty),
\end{equation}
by Plancherel's theorem and the Cauchy--Schwarz inequality, we have 
\[
\int_0^t\Bigl((\frakKeps\Lconv y)(s),y(s)\Bigr)_{L^2}\ds\geq \tilde{C}_{a,b} \|\frakKeps\Lconv y\|_{L^2_t(L^2(\Omega))}^2 
\]
for all $t \in (0,T)$. Therefore, if condition \eqref{ReFcalK} holds, the above estimate implies that assumption \eqref{assumption2} holds with the constant 
\begin{equation} \label{def_Ceps_Fourier}
\CfrakKeps=\tilde{C}_{a,b},
\end{equation}
where we highlight the dependence on $a$ and $b$.
An inspection of the last column in Table~\ref{table:kernels} and using the change of variables $\eps\omega \mapsto \omega$ yields
\[\tilde{C}_{a,b}=\inf_{\omega \in (0,\infty)} \Re(m(\omega)) = 2\pi \left(\frac{\rt}{\eps}\right)^{b-a} \inf_{\omega \in (0,\infty)} \Re((\imath\omega)^{b} + (\imath\omega)^{b-a} ).\]
The infimum is attained and it is bounded away from zero for $0<b \leq a \leq 1$. Indeed, for $a=b$, we have
\[
\tilde C_{a,a} = 2\pi \inf_{\omega \in (0,\infty)} \Re((\imath\omega)^{b}+1) = 2\pi\left(\inf_{\omega \in (0,\infty)}\cos(b\pi/2)\omega^b +1\right) =2\pi,
\]
while for $0<b<a\leq 1$,  
\begin{equation}\label{eq:c_a_b}
\tilde C_{a,b}  = 2\pi\left(\frac{\rt}{\eps}\right)^{b-a} \cos((b-a)\pi/2) \frac{a}{b} \left(\frac{(a-b) \cos((b-a)\pi/2)}{b\cos(b\pi/2)}\right)^{\frac{b-a}{b}}.
\end{equation}
Therefore, we conclude that the Mittag-Leffler kernels \eqref{ML_kernels_eps} with $0<b \leq a \leq 1$ satisfy assumption \eqref{assumption2} with the coercivity constant \eqref{def_Ceps_Fourier}. Recalling our discussion in Section~\ref{Sec:Gurtinpipkin_laws}, this allows us to cover the kernels originating from the Compte--Metzler laws GFE, GFE I, and GFE III; see Table~\ref{tab:ker_par_mod}. \\
\indent To verify \eqref{assumption2} for the Abel-type kernels and the kernel coming from law GFE II, we need an alternative approach.
~\\[2mm]
$\bullet$ {\bf Non-Fourier approach for completely monotone kernels.}  
An alternative approach of verifying \eqref{assumption2} relies on a coercivity property from~\cite[Lemma B.1]{kaltenbacher2021determining}, combined with a result on resolvents of the first kind, cf. \cite{gripenberg1990volterra}. This approach will be applicable on completely monotone kernels, such as the Abel kernel and the Mittag-Leffler functions when $b \geq a$. \\
\indent We first revisit~\cite[Lemma B.1]{kaltenbacher2021determining} to lower the regularity assumption on the memory kernel from $L^p(0,T)$ to $L^1(0,T)$. 
\begin{lemma}[Lemma B.1 in \cite{kaltenbacher2021determining} revisited]\label{lem:LemB1revisited}
	Given $T\in(0,\infty]$, let $\calK \in L^1(0,T)$. Furthermore, assume that $\calK \geq0$ on $(0,T)$ and that for all $t_0\in(0,T)$ it holds 
	\begin{align}
	\calK \in W^{1,1}(t_0, T), \quad \calK'|_{[t_0,T]} \leq 0 \  \text{ a.e.} 
	\end{align}
	Then
	\begin{equation}\label{idkB1}
	\intT (\calK \Lconv y_t, y)_{L^2}\ds \geq \frac 12 \calK \Lconv \|y\|_{\Ltwo}^2 (T) -\frac12 \intT \calK(s)\ds \|y(0)\|_{\Ltwo}^2
	\end{equation}    
	for all $y\in W^{1,1}(0,T; \Ltwo)$.
\end{lemma}
\begin{proof}
Similarly to the proof of \cite[Lemma B.1]{kaltenbacher2021determining}, we use an approximating sequence 
$\{\calK_n\}_{n\in \N}\subseteq W^{1,1}(0, T)$ defined as 
	\begin{align}
\calK_n (t)=\, \max\{\tilde{\calK}_n(t),w(t)/n\}
	\end{align}
	for some fixed positive weight function $w\in L^1(0,T)$ and
		\begin{align}
	\tilde{\calK}_n(t):=\, 1|_{[0,1/n]} \calK(1/n) + 1|_{[1/n,T]}(t) \calK(t).
	\end{align}
The identity \eqref{idkB1} can be shown to hold for $\calK_n$ in place of $\calK$ as in the proof of \cite[Lemma B.1]{kaltenbacher2021determining}. 
It thus remains to prove that $\calK_n$ converges to $\calK$ in the $L^1$ norm. 
To do so, we can rely on Lebesgue's dominated convergence theorem, using pointwise convergence of $K_n$ to $K$ and the fact that for any $n\in\mathbb{N}$, we have $0< \calK_n\leq \calK+ w \in L^1(0,T)$.
\end{proof}

Having in mind the verification of \eqref{assumption2}, we next extend the reasoning from \cite[Lemma 3.1]{kaltenbacher2022inverse} to a large family of completely monotone kernels $\frakK$. Completely monotone functions belong to $C^\infty(0,\infty)\subset W^{1,1}([t_0,T])$ for all $t_0>0$ and thus satisfy the regularity assumptions of Lemma~\ref{lem:LemB1revisited}.\\
\indent Let $\frakK \in L_{\textup{loc}}^1(\R^+)$ be completely monotone on $(0,\infty)$ and suppose $\frakK>0$ for almost all $t \in \R^+$.
By \cite[Theorem 5.5.4]{gripenberg1990volterra}, $\frakK$ has a resolvent of the first kind which is the sum of a point mass at zero and a completely monotone, locally integrable function:
\begin{equation} \label{def_r}
\mathfrak{r} = A\delta_0 + \mathfrak{f} \text{ such that }\frakK\Lconv\mathfrak{r}=1. 
\end{equation}
It is easy to see that $A$ has to be non-negative. Moreover, $A = 0$ if and only if $\displaystyle \lim_{t\to 0 } \frakK(t) =\infty$. In other words, due to the complete monotonicity, the $A$-term in \eqref{def_r} appears if and only if $\frakK \in L^\infty(0,T)$.

\begin{lemma}\label{lemma:someboundforA1}
	Let $\frakK \in L_{\textup{loc}}^1(\R^+)$ be completely monotone, nonconstant and a.e.\ positive ($\frakK>0$ for almost all $t \in \R^+$) and let \[\mathfrak{r} = A\delta_0 + \mathfrak{f}\] be its resolvent of the first kind. Then
	\begin{equation}
	\begin{aligned}\label{eq:lemma_ineq_bound}
	\int_0^{t} \intO \left(\frakK* y \right)(s) \,y(s)\dxs
	\geq&\, \frac 12 A\|\frakK \Lconv y\|_{\Ltwo}^2 (t)  + \mathfrak{f}(T) \,\,\frac12 \intt\|\frakK \Lconv y\|_{\Ltwo}^2\ds
	\end{aligned}
	\end{equation}
for all $	y\in L^2(0,t;L^2(\Omega))$ and $t \in (0,T)$, where for $T<\infty$ we have
$\mathfrak{f}(T)>0$. 
\end{lemma}
\begin{proof}
	We prove the inequality for $y\in C^\infty([0,T];L^2(\Omega))$; the statement follows then by density of $C^\infty([0,T])$ in $L^2(0,T)$. Applying Lemma~\ref{lem:LemB1revisited} on $\calK=\mathfrak{f}$ yields
	\begin{equation}
		\begin{aligned}
	&\int_0^t (\mathfrak{r}\Lconv v_t, v)_{L^2} \ds\\
	 =&\, \int_0^t \bigl((\mathfrak{f}\Lconv v_t, v)_{L^2} + A (v_t, v)_{L^2}\bigr) \ds \\
	\geq&\, \frac 12 A(\|v(t)\|_{\Ltwo}^2 - \|v(0)\|_{\Ltwo}^2) + \frac12 \big(\mathfrak{f}\Lconv \|v\|_{\Ltwo}^2(t) - \intt \mathfrak{f} \ds \|v(0)\|_{\Ltwo}^2\big),
	\end{aligned}
	\end{equation}
	for any $v\in W^{1,1}(0,T)$.
	By picking $v = \frakK \Lconv y$, 
	we have that $(\frakK \Lconv y)(0) = 0$ (since $y\in L^\infty(0,T)$), and subsequently that $\mathfrak{r}\Lconv v_t = y$.
	We then conclude that 
	\begin{align}
	\int_0^t (y, \frakK \Lconv y)_{L^2} \ds \geq &\frac 12 A\|\frakK \Lconv y\|_{\Ltwo}^2 (t)  + \frac12 (\mathfrak{f}\Lconv \|\frakK \Lconv y\|_{\Ltwo}^2)(t)
	\\
	\geq &\frac 12 A\|\frakK \Lconv y\|_{\Ltwo}^2 (t)  + \inf_{\tau\in(0,T)}\mathfrak{f}(\tau) \,\,\frac12 \intt\|\frakK \Lconv y\|_{\Ltwo}^2\ds,
	\end{align}
where by complete monotonicity $\displaystyle \inf_{\tau\in(0,T)}\mathfrak{f}(\tau)=\mathfrak{f}(T)$.

To prove that 
$\mathfrak{f}(T)>0$, 
we reason by contradiction. Suppose that $\mathfrak{f}(T)=0$. From the complete monotonicity of $\mathfrak{f}$ and its local absolute continuity it is clear that $\mathfrak{f}(t_1) = 0$ for all $t_1\geq T$. 
	We define $\tilde{T}=\inf\{t>0\ : \ \mathfrak{f}(t)=0\}$ and use that then for all $t_1>\tilde{T}$:
	$1=\frakK\Lconv\mathfrak{r}(t_1)= A \frakK(t_1)+ \int_0^{\tilde{T}} \frakK(t_1-s) \mathfrak{f}(s)\ds ,$
thus 
	\begin{equation}\label{diff0}
0= \int_0^{\tilde{T}} \Bigl(\frakK(\tilde{T}-s)-\frakK(t_1-s)\Bigr)\mathfrak{f}(s)\ds +A [\frakK(\tilde T)- \frakK(t_1)],
	\end{equation}
where by minimality in the definition of $\tilde{T}$ we have $f>0$ on $(0,\tilde{T})$. Moreover, by complete monotonicity of $\frakK$, we have that $A[\frakK(\tilde{T})-\frakK(t_1)]\geq0$.
Similarly, the factor $\frakK(\tilde{T}-s)-\frakK(t_1-s)$ is non-negative, and we conclude from \eqref{diff0} that it actually vanishes.
That is, for all $t_1>\tilde{T}$, $s\in [0,\tilde{T}]$ the identity $\frakK(\tilde{T}-s)=\frakK(t_1-s)$ holds, which by setting $s=\tilde{T}$, $t=t_1-s$ implies $\frakK(t)=\frakK(0)$ for all $t>0$.
This contradicts our assumption that $\frakK$ is nonconstant.
\end{proof}

Lemma~\ref{lemma:someboundforA1} implies an estimate of the type \eqref{assumption2} in case $\frakKeps$ is not constant. 
In the special case of constant kernels $\frakKeps\equiv c_\eps$ for some $c_\eps>0$, it is clear that assumption \eqref{assumption2} does not hold but instead the straightforward identity
	\begin{equation}
	\begin{aligned}
	\int_0^{t} \intO \left(\frakKeps* y \right)(s) \,y(s)\dxs
	=\, \frac{1}{2c_\eps} \|(\frakKeps \Lconv y)(t)\|_{\Ltwo}^2
	\end{aligned}
	\end{equation}
holds for all $	y\in L^2(0,t;L^2(\Omega))$ 
(note that $\|(\frakKeps \Lconv y)(0)\|_{\Ltwo}^2=0$).
The resulting PDE in that case is actually the inviscid Westervelt equation \eqref{ibvp_West_limit}
with $c^2$ and $f$ replaced by $c^2+c_\eps$ and $f-c_\eps \Delta u(0)$, respectively. 
Its well-posedness was already mentioned above to be a byproduct of Theorem~\ref{Thm:Wellp_2ndorder_nonlocal}.
Also letting $c_\eps$ converge to some limit, e.g., $c_\eps\to0$ is technically feasible due to the uniform energy bounds from Theorem~\ref{Thm:Wellp_2ndorder_nonlocal}, but probably not of high practical interest.

\medskip

	With $b\geq a$, the Mittag-Leffler function is completely monotone and satisfies assumptions of Lemma~\ref{lemma:someboundforA1}. Thus, recalling the kernels discussed in Section~\ref{Sec:Gurtinpipkin_laws}, we can use this result to verify the remaining case of the GFE II kernel, where $(a,b)=(\alpha, 1)$. For this particular kernel, the resolvent is given by
	$$\mathfrak{r} = \left(\frac{\rt}{\varepsilon}\right)^{1-\alpha}\left[\eps \delta_0 + \frac1{\varepsilon^{\alpha-1}}g_\alpha\right].$$
	One may indeed readily check that \[\mathscr{L}\left[{\left(\frac{\rt}{\eps}\right)^{\alpha-1}}\frac{1}{\eps} E_{\alpha,1}(-(\frac{t}{\eps})^\alpha) \Lconv \mathfrak{r}\right](z) = 1/z.\]
	\indent Similarly, the function $\eps \rt^{-\alpha}g_{\alpha}$ satisfies assumptions of Lemma~\ref{lemma:someboundforA1} and its resolvent is given by $\frac1\eps \rt^{\alpha}g_{1-\alpha}$.

\subsection*{Results of the verification} For convenience, we compile in Table~\ref{tab:ker_assu_veri} the results of the verification of assumptions \eqref{assumption1} and \eqref{assumption2} for memory kernels arising in the context of nonlinear acoustics. The last column in Table~\ref{tab:ker_assu_veri} provides the coercivity constant $\CfrakKeps$ in assumption \eqref{assumption2} and another argument in favor of assuming the ratio $\rt/\eps$ to be constant for laws GFE I and III. In that case, it is straightforward to check using Table~\ref{tab:ker_assu_veri} that $\CfrakKeps$ can be bounded uniformly from below by some $\cAtwo>0$ for $\eps \in(0,\beps)$.\\
\indent We point out that the theoretical framework developed in this paper, in particular pertaining to the continuity of the solution with respect to a parameter-dependent kernel, $\frakK_{\eps}$, in Theorem~\ref{Thm:Limit}, can be easily used to study the limiting behavior of the equations with respect to other physical parameters of interest. For instance, Theorem~\ref{Thm:Limit} can be employed in studying the limiting behavior of equations \eqref{ibvp_West_general} with the Mittag-Leffler kernels listed in Table~\ref{tab:ker_assu_veri} as $\alpha \nearrow 1$ . One can set $\eps=1-\alpha$ to frame the study within the previous theory. The key remaining component of such a limiting analysis would be to prove that assumptions \eqref{assumption1} and \eqref{assumption2} hold uniformly for $\alpha$ in a neighborhood of $1^-$. 

\begin{center}
	\begin{tabular}[h]{|l|c|c|c|}
		\hline
		&&&\\
		Examples of  kernels $\frakKeps$ in nonlinear acoustics&  \eqref{assumption1} &\eqref{assumption2}& $\CfrakKeps$\\[5pt]
		\hline\hline
		
		$\eps  {\rt^{-\alpha}} g_{\alpha}$	& \cmark & {\cmark}  & $\frac{1}{\eps}\rt^\alpha g_{1-\alpha}(T)$   \\[3mm]
		
		$\frac{1}{\eps} \exp\left(-\frac{t}{\eps}\right) $		& \cmark  & \cmark & $2\pi$  \\[3mm]
		
		GFE I: \		${\left(\frac{\rt}{\eps}\right)^{1-\alpha}}\frac{1 }{\eps^{{2\alpha-1}}}t^{2\alpha-2} E_{\alpha,2\alpha-1}(-(\frac{t}{\eps})^\alpha)$	& \cmark & \cmark & $\tilde C_{\alpha,2\alpha-1}$  \\[4mm]
		
		GFE II: \		${\left(\frac{\rt}{\eps}\right)^{\alpha-1}}\frac{1}{\eps} E_{\alpha,1}(-(\frac{t}{\eps})^\alpha)$	& \cmark & {\cmark} & $\rt^{1-\alpha} g_\alpha(T)$   \\[4mm]
		
		GFE III: \		${\left(\frac{\rt}{\eps}\right)^{1-\alpha}}\frac{1}{\eps^{{\alpha}}} t^{\alpha-1} E_{1,\alpha}(-(\frac{t}{\eps}))$ & \cmark & \cmark & $\tilde C_{1,\alpha}$   \\[4mm]
		GFE: \		$\frac{1}{\eps^\alpha} t^{\alpha-1}E_{\alpha,\alpha}(-(\frac{t}{\eps})^\alpha)$	& \cmark  & \cmark & $2\pi$  \\
		\hline
	\end{tabular}
	~\\[2mm]
	\captionof{table}{Kernels for  flux laws discussed in Section~\ref{Sec:AcousticModeling} and the assumptions they satisfy; $\tilde C_{\alpha,2\alpha-1}>0$ and $\tilde C_{1,\alpha}>0$ are constants defined in \eqref{eq:c_a_b} that can be made independent of $\eps$ (for GFE I and III, provided $\rt/\eps$ is constant).} \label{tab:ker_assu_veri}
\end{center} 

%% file: conclusion.tex
\section*{Conclusion and outlook}
In this work, we have investigated a family of nonlinear Westervelt equations with general dissipation of Djrbashian--Caputo-derivative type.  We rigorously studied them in terms of the local well-posedness and their limiting behavior with respect to the parameter $\eps$, which may be physically interpreted as the sound diffusivity or the thermal relaxation time. As we have seen, the limiting behavior is influenced by the dependence of the memory kernel $\frakKeps$ on $\eps$, and for this reason, we have considered different classes of kernels separately. 
 The framework developed here and leading to Theorem~\ref{Thm:Limit} allows extending the limiting study to other parameters of interest, as long as the parameter-dependent kernels satisfy assumptions \eqref{assumption1} and \eqref{assumption2}. \\
\indent Future work will be concerned with the study of the limiting behavior of quasilinear equations of Kuznetsov and Blackstock type given in \eqref{Blackstock_nonlocal} and \eqref{Kuznetsov_nonlocal}, respectively, where we expect that the kernel assumptions will need to be further tailored to the needs of each of these models. It would also be of interest to extend the limiting analysis to settings with practical boundary conditions, such as Neumann or absorbing boundary conditions. We expect that the main ideas from the current energy arguments can be adjusted but with a higher level of technicality due to the lack of Poincaré-Friedrichs inequality. Finally, the important question of whether global results in time can be established when the limiting problems contain dissipation (such as \eqref{ibvp_West_fractional_limit_tau}) remains open.

%% file: appendix_semidiscrete.tex
\begin{appendices} 
	\section{Unique solvability of the nonlocal semi-discrete problem}	\label{Appendix:Semi-discrete}
	We present here the proof of existence of a unique semi-discrete approximation of \eqref{ibvp_West_general_lin}. As in~\cite{kaltenbacher2022parabolic}, we may  approximate the solution by
	\begin{equation}
		\begin{aligned}
			u^{\eps, n}(x,t) =\sum_{i=1}^n \xi^n_i(t)v_i(x),
		\end{aligned}
	\end{equation}
	where $\{v_i\}_{i=1}^\infty$ are smooth eigenfunctions of the Dirichlet-Laplacian operator.
	\\ \indent With  $V_n=\textup{span}\{v_1, \ldots, v_n\}$,  the semi-discrete problem is given by
	\begin{equation} \label{semi-discrete}
		\begin{aligned}
			\begin{multlined}[t] ((1+2k \phi) \utt^{\eps, n}- c^2 \D u^{\eps, n} - \frakKeps*\D u_t^{\eps, n}+2k \phi_t u_t^{\eps,n}-f, v_j) = 0 \end{multlined}
		\end{aligned}
	\end{equation}
	for all $j=1, \ldots, n$, with approximate initial conditions $(u^{\eps, n}, \ut^{\eps, n})\vert_{t=0}=(u_{0}^n, u_{1}^n)$ taken as $\Ltwo$ projections of $(u_0, u_1)$ onto $V_n$. 
	With $\boldsymbol{\xi}=[\xi^n_1 \ \xi^n_2 \ \ldots \ \xi^n_n]^T$, the approximate problem can be rewritten in matrix form
	\begin{equation} \label{matrix_eq}
		\begin{aligned}
			\mathbb{M}_{\aaa}(t)\boldsymbol{\xi}_{tt}+ c^2\mathbb{D} \boldsymbol{\xi}+ \mathbb{D}\, \frakKeps*\boldsymbol{\xi}_t+\mathbb{M}_{\aaa_t}(t)\bfxi_t= \boldsymbol{f},
		\end{aligned}
	\end{equation}
	where the entries of matrices $\mathbb{M}_{\aaa}(t)=[\mathbb{M}_{\aaa, ij}]$, $\mathbb{M}_{\aaa_t}(t)=[\mathbb{M}_{\aaa_t, ij}]$, $\mathbb{D}=[\mathbb{D}_{ij}]$, and vector $\boldsymbol{f}=[\boldsymbol{f}_{j}]$ are given by
	\begin{equation} \label{matrices}
		\begin{aligned}
			& \mathbb{M}_{\aaa, ij}(t)= ((1+2k \phi)  v_i, v_j)_{L^2}, \quad && \mathbb{M}_{\aaa_t, ij}(t)= (2k \phi_t  v_i, v_j)_{L^2},   \\
			& \mathbb{D}_{ij}= -(\Delta v_i,  v_j)_{L^2},\quad  && \boldsymbol{f}_i(t)=(f(t), v_i)_{L^2}.
		\end{aligned}
	\end{equation}
	If	we  introduce the vectors of coordinates of the approximate initial data in the basis: 
	\[
	\boldsymbol{\xi}_0=[\xi^n_{0,1} \ \xi^n_{0,2} \ \ldots \ \xi^n_{0,n}]^T, \quad \boldsymbol{\xi}_1=[\xi^n_{1,1} \ \xi^n_{1,2} \ \ldots \ \xi^n_{1,n}]^T,
	\]
	then by setting $\bfmu=\boldsymbol{\xi}_{tt}$, we have
	\begin{equation} \label{eq_xi}
		\bfxi_t(t)=1\Lconv\bfmu+\bfxi_1, \quad	\boldsymbol{\xi}(t)=\boldsymbol{\xi}_0 +t \boldsymbol{\xi}_1+1*1*\bfmu.
	\end{equation}
	Therefore, the semi-discrete problem can be rewritten as
	\begin{equation}
		\begin{aligned}
			\begin{multlined}[t]
				\mathbb{M}_{\aaa}(t)\bfmu+ c^2  \mathbb{D}\left(\boldsymbol{\xi}_0 +t \boldsymbol{\xi}_1+1*1*\bfmu\right)+ \mathbb{D}\,\frakKeps *(1*\bfmu+\bfxi_1)+ \mathbb{M}_{\aaa_t}(t)(1*\bfmu+\bfxi_1) =\boldsymbol{f}. \end{multlined}
		\end{aligned}
	\end{equation}
	Since $\mathbb{M}_{\aaa} \in L^\infty(0,T)$ is positive definite due to assumption \eqref{non-deg_phi}, the semi-discrete problem can be further seen as a system of Volterra integral equations: 
	\begin{equation}
		\begin{aligned}
			\bfmu+ \boldsymbol{K} * \bfmu = \boldsymbol{\tilde{f}},
		\end{aligned}
	\end{equation}
	with
	\[
	\boldsymbol{K}= \mathbb{M}_{\aaa}^{-1}(t)\left\{c^2 \mathbb{D}1*1 + \mathbb{D}\,\frakKeps *1+\mathbb{M}_{\aaa_t}(t)\right\}
	\]
	and
	\[
	\boldsymbol{\tilde{f}}=\mathbb{M}_{\aaa}^{-1}(t)\left\{-c^2 \mathbb{D}\left(\boldsymbol{\xi}_0 +t \boldsymbol{\xi}_1\right)- \mathbb{D}\,\frakKeps *\bfxi_1-\mathbb{M}_{\aaa_t}(t) \bfxi_1+\boldsymbol{f}(t)\right\}.
	\]
	By the existence theory for systems of Volterra integral equations of the second kind~\cite[Ch.\ 2, Theorem 4.5]{gripenberg1990volterra}, there is a unique solution $\bfmu \in L^\infty(0,T)$. Combined with initial data, from $\bfxi_{tt}=\bfmu$, we can then conclude that there exists a unique $\bfxi \in W^{2, \infty}(0,T)$ and thus $u^{n, \eps} \in W^{2, \infty}(0,T; V_n)$. 
\end{appendices}